\documentclass{amsart}

\usepackage{amsmath}
\usepackage{amssymb}
\usepackage{enumerate}
\usepackage{verbatim}
\usepackage{xypic}

\numberwithin{equation}{section}

\theoremstyle{theorem}
\newtheorem{thm}[equation]{Theorem}%[section]
\newtheorem{lemma}[equation]{Lemma}

\newtheorem{prop}[equation]{Proposition}
\newtheorem{coro}[equation]{Corollary}
\newtheorem{ex}[equation]{Example}

\theoremstyle{definition}

\theoremstyle{remark}
\newtheorem{remark}[equation]{Remark}

\newcommand{\exref}[1]{Ex\-am\-ple \ref{#1}}
\newcommand{\thmref}[1]{Theo\-rem \ref{#1}}

\newcommand{\lemref}[1]{Lem\-ma \ref{#1}}
\newcommand{\propref}[1]{Prop\-o\-si\-tion \ref{#1}}
\newcommand{\corref}[1]{Cor\-ol\-lary \ref{#1}}
\newcommand{\figref}[1]{Fig\-ure \ref{#1}}

\newcommand{\remref}[1]{Re\-mark \ref{#1}}

\providecommand{\divides}{\mid}
\providecommand{\ndivides}{\nmid}

\providecommand{\normaleq}{\unlhd}

\DeclareMathOperator{\Bi}{Bi}
\DeclareMathOperator{\Adj}{Adj}

\DeclareMathOperator{\End}{End}

\DeclareMathOperator{\GL}{GL}
\DeclareMathOperator{\Sp}{Sp}
\DeclareMathOperator{\Isom}{Isom}
\DeclareMathOperator{\Gal}{Gal}
\DeclareMathOperator{\Grp}{Grp}
\DeclareMathOperator{\Aut}{Aut}
\DeclareMathOperator{\Irr}{Irr}
\DeclareMathOperator{\Cent}{Cent}

\DeclareMathOperator{\Qd}{Qd}

\newcommand{\lt}[1]{{#1}^{\Lsh}}
\newcommand{\rt}[1]{{#1}^{\Rsh}}

\newcommand{\ct}[1]{{#1}^{\uparrow}}

\title{Isomorphism in expanding families of indistinguishable groups}
\author{Mark L. Lewis}
\address{
    Department of Mathematical Sciences\\
    Kent State University\\
    Kent, OH 44242
}
\email{lewis@math.kent.edu}
\author{James B. Wilson}
\address{
    Department of Mathematics\\
    Colorado State University\\
    Fort Collins, CO 80523\\
}
\email{jwilson@math.colostate.edu}
\date{\today}

\begin{document}

\begin{abstract}
For every odd prime $p$ and every integer $n\geq 12$, there is a Heisenberg group of order $p^{5n/4+O(1)}$ that has $p^{n^2/24+O(n)}$ pairwise nonisomorphic quotients of order $p^{n}$. Yet, these quotients are virtually indistinguishable.  They have isomorphic character tables, every conjugacy class of a non-central element has the same size, and every element has order at most $p$.  They are also directly and centrally indecomposable and of the same indecomposability type.  Nevertheless, there is a polynomial-time algorithm to test for isomorphisms 
between these groups.
\end{abstract}

\maketitle

%---------------------
\section{Introduction}

Deciding that two groups are isomorphic is a clear task: exhibit an invertible homomorphism between the groups.  On the other-hand, understanding why two groups are non-isomorphic can take many different forms, and in this paper we demonstrate how little we know about non-isomorphism.  To illustrate the situation,  we can prove that the dihedral group $D_{2^n}$ of order $2^n$ is non-isomorphic to the quaternion group $Q_{2^n}$ of order $2^n$ by checking that no mapping of generators for $D_{2^n}$ to generators for $Q_{2^n}$ extends to a homomorphism.  Instead, we usually report on some group isomorphism invariant, e.g. that $D_{2^n}$ has many elements of order $2$ whereas $Q_{2^n}$ has only one.  The latter is both informative and easier to prove.

In this article, we produce a family of groups each with size $p^n$ that have $p^{O(n^2)}$ different isomorphism types, but for which no obvious isomorphism invariant presents itself to distinguish a pair of groups from the family.  Yet, given a pair of groups from the family we can efficiently (in polynomial time) test if they are isomorphic.  If the algorithm does not produce an isomorphism, then we have proved that the groups are non-isomorphic.  Such a proof is as informative as a proof that $D_{2^n}\not\cong Q_{2^n}$ by exhausting all possible functions between them.  Such ``zero-knowledge'' non-isomorphism tests rightfully raise suspicion.

The family we produce is one of many, and it arose out of a larger study of Camina groups; we will say more about this in Section \ref{sec:closing}.  Though our family is very simple to describe, it also lies within the class of groups for which isomorphism appears most difficult to understand.  As a consequence, the group theory aspects of the proof are modest and straight-forward, but most of the proof is accomplished by use of bilinear maps, rings with involutions, and tensor products.  As these are not yet common tools for groups, we survey in Section \ref{sec:survey} the main ideas of these tools.

\subsection{Main results}
A group $H$ is a {\em generalized Heisenberg group} if there is a field $K$ and an integer $m$ such that $H$ is isomorphic to
\begin{align}\label{eq:gen-Hei}
 H_m(K) &  = \left\{ \begin{bmatrix} 1 & u & s \\ 0 & I_m & v^t \\ 0 & 0 & 1 \end{bmatrix}:
 s\in K, u,v\in K^m\right\}
\end{align}
When $m = 1$ we call $H$ a \emph{Heisenberg group}.   The family of groups in which we are interested are the nonabelian quotients of $H$.

First, a generalized Heisenberg group $H$ has an extraordinary number (compared to $|H|$) of nonisomorphic quotients of a fixed order.    We prove:

\begin{thm}\label{thm:main}
For every prime $p>2$ and every integer $n\geq 12$, there is a generalized Heisenberg group (in fact a Heisenberg group) of order $p^{5n/4+O(1)}$ that has $p^{n^2/24+O(n)}$ isomorphism classes of quotient groups that have
order $p^n$.
\end{thm}

It is not surprising that a group will have a large number of nonisomorphic quotients (consider free groups).  For comparison, Higman \cite[Section 2]{Higman:enum} created groups $F_N$ having $N^{O(\log_p^2 N)}$ distinct isomorphism classes appearing as quotients of $F_N$ and with size $N = p^n$; yet, $F_N$ has size $N^{O(\log_p N)}$.  The surprise in \thmref{thm:main} is that we obtain $N^{O(\log_p N)}$ distinct isomorphism classes of groups of size $N$ from a group of size as small as $N^{1.2+O(1/\log_p N)}$.  As these quotients are so large compared to the size of the parent group, they must have an extraordinary number of relations in common, but yet, they still display enormous diversity.

Despite the great number of isomorphism classes guaranteed by \thmref{thm:main}, our second result claims that we can relatively simply determine when two quotients of a generalized Heisenberg group are isomorphic. Algorithms to test for an isomorphism between general groups of order $N$ return an answer in $N^{\log_p N+O(1)}$-time \cite{Miller:nlogn}, where $p$ is the smallest prime dividing $N$, and where \emph{time} indicates an upper bound on the number of steps a routine performs.  It is an important open problem to determine if isomorphism testing of groups can be done in polynomial time in the order $N$ of the groups, but progress in this direction has been slow.  Amongst the hardest cases are the groups of order $N = p^n$, where $p$ is a prime, and having nilpotence class 2, such as quotients of generalized Heisenberg groups.  Indeed, for these groups, the most advanced method, known as the \emph{nilpotent quotient algorithm}, runs in time $N^{\log_c N} = p^{n^2/c'+O(n)}$, where $c$ and $c'$ depend only on $p$; see \remref{rem:nil-q-algo}.  For a survey of group isomorphism algorithms see \cite{Babai:iso,CH:iso,OBrien:iso}.

The algorithm in our next theorem works with groups given by generators (as permutations or matrices) and also with groups specified by black-box polycyclic presentations,\footnote{We say `black-box' here because multiplication in polycyclic groups is in the worst case exponential in the length of the presentation.  However, in practice operating in polycyclic groups is amongst the most efficient means for working with $p$-groups. So we regard the cost of multiplication as an acceptable constant and measure efficiency in that setting in terms of number of group operations.}  and so polynomial time in these contexts is a function of the these very terse input methods.  Hence, our algorithm represents an exponential improvement over all other known isomorphism tests that apply to these $p$-groups.  We had originally proved it only in the context of permutation representations.  We are indebted to L. Ronyai for an elegant adaptation (\lemref{lem:Ronyai-trick}) that extends our earlier algorithm to the remaining common input methods for groups.  We prove:

\begin{thm}\label{algo:main}
There are algorithms that determine
\begin{enumerate}[(i)]
\item if a group $G$ (given by permutations, matrices, or a black-box polycyclic presentation) is an epimorphic image of an odd order generalized Heisenberg group, and if so, then returns an epimorphism $H_m (K) \to G$ with $|H_m (K)|$ as small as possible, and
\item if two groups, that are epimorphic images of odd order generalized Heisenberg groups, are also isomorphic.
\end{enumerate}

The algorithms are deterministic polynomial-time in $\log |G|+p$ and Las Vegas\footnote{Las Vegas algorithms always return correct answers but with a user specified probability of $\varepsilon>0$, they may abort without an answer.} polynomial-time in $\log |G|$ (owing to the implicit need to factor polynomials over finite fields of characteristic $p$).
\end{thm}

In our third and final result, we list our failures to distinguish the quotients of odd order generalized Heisenberg groups $H$ by traditional means.  In light of \thmref{thm:main}, one might expect that two quotients $G_1$ and $G_2$ of $H$ with the same order $p^n$ will be considerably distinct as groups, and in view of \thmref{algo:main} (ii), it would likely be straightforward to describe these differences.  Unfortunately, the  algorithm of \thmref{algo:main} (ii) does not appear to produce a group-theoretic property to characterize each isomorphism class.

Because of \thmref{algo:main} (i), we are concerned only with the differences between quotients $G_1$ and $G_2$ of a common generalized Heisenberg group $H = H_m (K)$ for which $|H|$ is as small as possible.  We say such quotients are \emph{indigenous} to $H$.  So our effort is to find isomorphism invariants for indigenous quotients $G_1$ and $G_2$ of $H$.  We also assume $|G_1| = |G_2|$, but amazingly that assumption appears to force a great number of typically discerning isomorphism invariants to be the same for both $G_1$ and $G_2$. Every non-trivial element of $G_1$ and $G_2$ has order $p$.  Also, $G_1$ and $G_2$ have isomorphic character tables, indeed the centralizer of every non-central element has the same size.   Next, we consider recent advances on decompositions of $p$-groups as in \cite{Wilson:unique-cent}, but we find indigenous quotients are directly and centrally indecomposable and of the same `type' of indecomposability.  With some modest constraints on the $|G_i|$ relative to $|H|$, we retain the large number of isomorphism types described in \thmref{thm:main} but also constrain the automorphism groups of the $G_i$ to have identical subgroups $C_i=C_{\Aut G_i}(G'_i)$ and furthermore, $\Aut G_i/C_i$ can take at most $2d(K)$ different values where $d(K)$ is the number of divisors of $\log_p |K|$.  In fact, if $\log_p |K|$ is prime, we have at most 2 types of automorphism groups possible.  The isomorphism invariants just described are often quite powerful even in difficult contexts involving $p$-groups of class $2$, e.g.  \cite[pp. 143--144]{Verardi} \& \cite[p. 99]{Elation}.  Therefore, we found it startling to have no use for them on such a large family of groups.

We hope we have illustrated the need for creative alternative structural properties that will apply to $p$-groups of class $2$.  Ideally, these new properties would be easily computed (say in polynomial time) and would lead to isomorphism invariants that would help us understand isomorphism of $p$-groups in broader contexts.  Admittedly, the interest in quotients of generalized Heisenberg groups is narrow, but we use these as an example of an entirely obvious family of groups for which the group isomorphism problem presents some of its most puzzling properties.

\subsection{Survey}\label{sec:survey}

Because of \thmref{algo:main}, we cannot assume that a group is specified in any manner relating to the natural definition of a Heisenberg group.  Therefore, our first step is to uncover properties of a group $G$ that determine when it is a generalized Heisenberg group and when it is an epimorphic image of a generalized Heisenberg group.  To obtain a usable algorithm, we also take care to involve properties of $G$ that can be computed efficiently.

The first step uses the commutation map of a $p$-group of class $2$.  This map $b =\Bi(G): G/Z(G) \times G/Z(G) \to G'$ assigns $b(Z(G)x,Z(G)y) = [x,y]$.  Baer observed that $b$ is biadditive.  Using this observation, we are able to translate our group questions to linear algebra and classical geometry.  From this result, we can identify when $G$ is a generalized Heisenberg group by determining the largest commutative ring $K = \Cent (b)$ for which $b$ becomes $K$-bilinear.  We show $G$ is a generalized Heisenberg group if and only if $K$ is a field and $b$ is an alternating nondegenerate $K$-form  (\thmref{thm:rec-Hei}).

To recognize epimorphic images $G$ of a generalized Heisenberg group $H$, we first remark that $b=B(G)$ factors through $\Bi(H)$.   To construct a suitable group $H$ from $G$, we construct $A = \Adj (b)$ as the largest ring over which $b$ factors through the tensor product $\otimes_A : G/Z(G) \times G/Z(G) \to (G/Z(G)) \otimes_A (G/Z(G))$ -- that requires that $A$ be defined to act on the right and left of $G/Z(G)$ and so $A$ is equipped with an anti-isomorphism of order at most $2$, i.e. an \emph{involution}.  Using properties of simple rings with involutions and their representations, we show that for epimorphic images of Heisenberg groups, the tensor product $\otimes_{A}$ is a nondegenerate alternating $F$-form for the center $F$ of $A$.  Indeed, $G$ is an epimorphic image of $H_m (F)$ where $2m=\dim_F (G/Z(G))$; in fact, $G$ is indigenous to $H_m(F)$ (\thmref{thm:rec-q-Hei}).  

Our tools so far are computable and rely mostly on linear algebra techniques and factoring polynomials.  In particular we have described enough already to prove \thmref{algo:main}(i).

The next crucial step is to show that when $G_1$ and $G_2$ are indigenous quotients of a generalized Heisenberg group $H$, then every isomorphism $\phi : G_1 \to G_2$ lifts to an automorphism of $H$ (\thmref{thm:lift-iso}).  This is done by using $\phi$ to induce a pseudo-isometry $(\varphi;\ct{\varphi})$ from $b_1=\Bi(G_1)$ to $b_2=\Bi(G_2)$ which is then extended to a pseudo-isometry $(\varphi;\ct{\Phi})$ between the tensors $\otimes_{\Adj (b_1)}$ and $\otimes_{\Adj (b_2)}$ (pseudo-isometry is the appropriate equivalence relation between alternating biadditive maps).  As the $G_i$ are indigenous to $H$, $\Bi(H)$ is pseudo-isometric to both $\otimes_{\Adj (b_1)}$ and $\otimes_{\Adj (b_2)}$, and so, we can obtain an automorphism of $H$ from $(\varphi;\ct{\Phi})$.

Finally, we prove our main theorems by considering the well-known structure of the automorphism group of a generalized Heisenberg group $H$.  From the isomorphism lifting property, two epimorphic images of $H$ are isomorphic if and only if their kernels lie in the same $(\Aut H)$-orbit.  As these kernels can be identified with $\mathbb{Z}/p$-subspaces of a finite field $K$, this amounts to understanding the $(\Gal(K)\ltimes K^{\times})$-orbits of the $\mathbb{Z}/p$-subspaces of $K$.  Each of these orbits is small, and so, there are many obits.  That explains the many isomorphism types in \thmref{thm:main}.  We use Ronyai's modification to test when two subspaces lie in the same orbit and so produce a very efficient test of isomorphism; \thmref{algo:main} (ii).

\subsection{Outline}

Section \ref{sec:back} gives background and Section \ref{sec:rec} deals with recognizing quotients of generalized Heisenberg groups. We prove our main theorems in Section \ref{sec:main}.  Section \ref{sec:invariants}  demonstrates a list of typically sensitive group isomorphism invariants which here are of no use.  Section \ref{sec:closing} considers $2$-groups and a problem of Brauer.

%--------------------------------------------
\section{Background}\label{sec:back}

Throughout $p$ will denote an odd prime.  All our groups, rings, and modules will be finite unless context makes this obviously false.  We will use the following standard group theory notations.  For elements $g,h\in G$, write $g^h = h^{-1}g h$, $[g,h] = g^{-1} g^h$, and $g^G = \{g^h:h\in G\}$. To fit these conventions, homomorphisms $\varphi : G \to H$ are evaluated as $g \varphi$, for $g \in G$, and all other functions are, as usual, on the left.  Given subgroups $H,K \leq G$, set $[H,K] = \langle [h,k] : h \in H, k\in K\rangle$.
Also, for a subset $S\subseteq G$, we write $C_G(H) = \{h\in G:\forall g\in S, [g,h] = 1\}$ to denote the \emph{centralizer} of $S$ in $G$.  Call $G' = [G,G]$ the \emph{commutator} subgroup of $G$, and $Z (G) = C_G (G)$ the \emph{center} of $G$.  We say that $G$ is \emph{nilpotent of class $2$} if $1 < G'\leq Z (G) < G$.  A group $G$ has \emph{exponent $p$} if $G^p = \langle g^p : g\in G \rangle$ is trivial.

%----------
\subsection{Bimaps}

In this work, we will typically need $k$ to be a finite field, but for the moment we require only that $k$ be a commutative unital ring and that $U$, $V$, and $W$ be $k$-modules.  We write $\End_k U$ for the ring of $k$-linear endomorphisms of $U$ and $\GL_k (U)$ for the group of $k$-linear automorphisms of $U$.  In cases were $k$ is omitted from the notation, it should be assumed to be the integers, which in most contexts could further reduce to the appropriate prime subfield $\mathbb{Z}/p$.

A  \emph{$k$-bimap} is a function $b : U \times V  \to
W$ of $k$-modules $V$ and $W$ with
\begin{align*}
  b (u+rx,v) & = b (u,v) + rb (x,v) & ( \forall u,x \in U,\forall v \in V,\forall r \in k)\\
  b (u,v+rx) & = b (u,v) + rb (u,x) & ( \forall u \in U,\forall v,x \in V,\forall r \in k).
\end{align*}
We say $b$ is {\em alternating} if $U = V$ and $b (u,u) = 0$ for all elements $u \in V$. Every $k$-bimap is also a $\mathbb{Z}$-bimap (even a $\mathbb{Z}/e$-bimap where $e$ annihilates $U \times V \times W$).  We say that $b$ is a \emph{$k$-form} if $W$ is a cyclic $k$-module.  Given $X,Y \subseteq V$, define $b (X,Y) = \langle b (x,y) : x \in X, y \in Y \rangle$.  For a $k$-linear map $\varphi : W \to Z$, we use $b\varphi$ for the bimap $V \times V \to Z$ defined as follows:
\begin{align*}
  (b\varphi) (u,v) & = b (u,v) \varphi & (\forall u,v \in V).
\end{align*}
In general we say a bimap $c:U\times V\to X$ \emph{factors through} $b$ if there is a $\phi:W\to X$ such that $c=b\phi$.
The left and right \emph{radicals} of $b$ are the submodules
$U^{\bot}  = \{ v \in V : b (U,v) = 0 \}$ and $V^{\top}  = \{u \in U : b (u,V) = 0 \}$.
Say that $b$ is \emph{nondegenerate} if $U^{\bot} = 0$ and $V^{\top} = 0$.  If $b$ is alternating, then $U^{\top} = V^{\bot}$.

A pair $b : U \times V \to W$ and $b' : U' \times V' \to W'$ of $k$-bimaps are \emph{(strongly) $k$-isotopic} if there is a triple
$(\lt{f} : U \to U', \rt{f} : V \to V'; \ct{f} : W \to W')$ of $k$-linear isomorphisms such that
\begin{align*}
  b (u,v)\ct{f} & = b'(u\lt{f}, v\rt{f}) & (\forall u,v \in V).
\end{align*}
(There is a notion of weak isotopism which will not be needed here.)
If $U = V$ and $U' = V'$, then we can consider a $k$-{\em pseudo-isometry} which is a $k$-isotopism $(\lt{f},\rt{f};\ct{f})$ where
$\lt{f} = \rt{f} =: f$.  We abbreviate $(\lt{f},\rt{f};\ct{f})$ by $(f;\ct{f})$ in that instance, but we remark that $\ct{f}$ is not completely determined by $f$ unless $W = b (V,V)$.  Finally,  if $W = W'$, then we define an \emph{isometry} as a pseudo-isometry $(f;\ct{f})$ with $\ct{f} = 1_W$.  In particular, we have the following natural groups of
pseudo-isometries and isometries for a $k$-bimap $b : V \times V \to W$:
\begin{align*}
    \Psi\Isom_k(b) & = \{(f;\ct{f})\in \GL_k(V) \times \GL_k(W):\forall u,v\in V,
        b (uf,vf) = b (u,v)\ct{f}\}\\
    \Isom_k(b) & = \{(f;\ct{f})\in \Psi\Isom_k(b): \ct{f} = 1\}\normaleq \Psi\Isom_k(b).
\end{align*}

\begin{remark}\label{rem:one-alt-form}
Every alternating nondegenerate $K$-form $j:V \times V \to K$ has a $K$-basis $\{e_1,\dots, e_m,f_1,\dots,f_m\}$ such that $j(e_i,e_j) = 0 = j(f_i,f_j)$ and $j(e_i,f_j) = \delta_{ij}$, for all $i$ and $j$ in $\{1,\dots,m\}$.  Hence, there is only one $K$-pseudo-isometry class of nondegenerate alternating $K$-form and  we take the bimap of \eqref{eq:alt} as a canonical representative from that class, defined by
\begin{align}\label{eq:alt}
    j(u,v) & = u\begin{bmatrix} 0 & I_m\\ -I_m & 0 \end{bmatrix} v^t & (\forall u,v \in K^{2m}).
\end{align}
\end{remark}

%-------------
\subsection{Baer's correspondence}\label{sec:Baer}
We work with odd $p$-groups by means of bimaps as introduced by Baer \cite{Baer:class-2}.  This method is the first approximation of the now well-established use of the Mal'cev-Kaloujnine-Lazard correspondence (sometimes inadequately referred to as the Baker-Campbell-Hausdorff formula); see \cite[Section V.5]{Jacobson:Lie} and \cite[Section 10]{Khukhro} for details.  In Section \ref{sec:2-groups}, we make a modest effort to extend this correspondence for use with Heisenberg $2$-groups.

Associated to each group $G$ of nilpotence class $2$ (without restriction on its order) is a function $b = \Bi (G):G/Z(G) \times  G/Z(G) \to G'$ where
\begin{align}
    b (Z(G)x,Z(G)y) & = [x,y] & (\forall x,y \in G ).
\end{align}
Baer showed that $b$ is an alternating nondegenerate $\mathbb{Z}$-bimap and now we write it additively.  If the exponent of $G$ is a prime $p$ (or more generally, if $G^p\leq Z(G)$ and $(G')^p = 1$), then $b$ is a $\mathbb{Z}/p$-bimap. We say that groups $G_1$ and $G_2$ of nilpotence class $2$ are \emph{isoclinic} if $\Bi (G_1)$ and $\Bi (G_2)$ are $\mathbb {Z}$-pseudo-isometric.  (This agrees with the usual broader meaning of isoclinism introduced by P. Hall.)  When $G_1$ and $G_2$ are isomorphic, they are immediately isoclinic.  Yet, $D_8$ and $Q_8$ are isoclinic but nonisomorphic groups.

\begin{ex}\label{ex:j-Hei}
If $H = H_m(K)$, then
\begin{align*}
	 H' = Z(H) = \left\{\begin{bmatrix} 1 & 0 & s \\ 0 & I_m & 0 \\ 0 & 0 & 1 \end{bmatrix} : s \in K\right\},
\end{align*}
and $\Bi (H)$ is an alternating nondegenerate $K$-form.
\end{ex}

In particular, $\Bi (H)$ is $\mathbb {Z}$-pseudo-isometric to $j:K^{2m} \times K^{2m} \to K$ in \eqref{eq:alt}.  (Later in Section \ref{sec:centroid} we show $\Bi (H)$ is a natural $K$-bimap and as such is $K$-pseudo-isometric to $j$, but for now $\Bi (H)$ is defined only as a $\mathbb {Z}$-bimap.)

Baer's bimap (above) establishes a natural correspondence between certain nilpotent groups of class $2$ and alternating bimaps.  If $b : V \times V \to W$ is an alternating $\mathbb{Z}[1/2]$-bimap, then define the corresponding Baer group $G=\Grp(b)$ for $b$ as the set $V \times W$ equipped with the product:
\begin{align}\label{def:Baer-group}
	(u;s) (v;t ) & = \left(u+v;s+t+\frac{1}{2}b (u,v)\right)
\end{align}
This is a group with familiar properties including:
$\forall u,v\in V$, $\forall s,t\in W$,$\forall e\in\mathbb{Z}$,
\begin{align}
\label{eq:exp}
    (u;s)^e & = \left(eu; es \right), \textnormal{ and }\\
\label{eq:comm}
    [(u;s),(v;t)] & = (0; b (u,v) ).
\end{align}
Hence, the center and commutator subgroups are as follows:
\begin{align}
    G'& = 0 \times b (V,V)\leq 0 \times W \leq V^{\bot(b)} \times  W  = Z(G).
\end{align}
In particular, $G$ is nilpotent of class $2$.  Notice that every $\mathbb{Z}$-pseudo-isometry $(\varphi;\hat{\varphi})$ from $b$ to another bimap $b':V' \times V' \to W'$ induces an isomorphism $(u;s)\mapsto (u\varphi;s\hat{\varphi})$ from $\Grp(b)$ to $\Grp(b')$.
Hence, if $b$ is nondegenerate and $W = b (V,V)$, then \eqref{eq:comm} implies that $b$ and $\Bi (\Grp(b))$ are naturally pseudo-isometric (by identifying $W$ with $0 \times W = \Grp(b)' = Z(\Grp(b))$ and $V$ with $(V \times W)/(0 \times  W)$).  Also,
for nilpotent groups $G$ of class $2$ for which $G/Z(G)$ and $G'$ have no $2$-torsion, it follows that $G$ is isoclinic to $\Grp(\Bi (G))$.  When $G^p = 1$ (which implies $p>2$) and $G' = Z(G)$,  it is possible to upgrade isoclinism to isomorphism.

\begin{prop}[Baer, 1939]\label{prop:Baer-correspondence}
If $G$ is a $p$-group where $1 = G^p < G' = Z(G) < G$ (so $p>2$), then every transversal $\ell:G/G' \to G$ with $0\ell = 1$ induces an isomorphism $\varphi_{\ell}:G \to \Grp(\Bi (G))$.
Also,
\begin{align*}
	\Aut G\cong \Psi\Isom_{\mathbb{Z}/p}(\Bi (G)) \ltimes_{\tau} \hom_{\mathbb{Z}/p}(G/Z(G),G').
\end{align*}
where for each $f\in \hom_{\mathbb{Z}/p}(G/Z(G),G')$ and each $(\varphi;\ct{\varphi})\in\Psi\Isom_{\mathbb{Z}/p}(\Bi (G))$, $(f)(\varphi;\ct{\varphi})\tau = \varphi^{-1} f\ct{\varphi}$.
Specifically, if $G = \Grp(b)$ for an alternating $\mathbb{Z}/p$-bimap $b:V \times V \to W$ with $W = b (V,V)$, then
\begin{enumerate}[(i)]
\item for all $(\varphi;\ct{\varphi})\in \Psi\Isom_{\mathbb{Z}/p}(\Bi (G))$ and all $(u;s)\in V \times  W$,
$(u;s)^{(\varphi;\ct{\varphi})} = (u\varphi;s\hat{\varphi})$, and
\item after canonically identifying $V$ with $G/Z(G) = (V \times  W)/(0 \times W)$ and $W$ with $G' = 0 \times W$, for all $\tau\in\hom(V,W)$ and all $(u;s)\in V \times  W$, $(u;s)^{\tau}  = (u;s+u\tau)$.
\end{enumerate}
\end{prop}
\begin{proof}
For the isomorphism of $G$ to $\Grp(\Bi (G))$ see \cite[Proposition 3.10]{Wilson:unique-cent}.
For the remaining properties, observe $\Psi\Isom(\Bi (G))$ embeds in $\Aut \Grp(\Bi (G))$ as argued above.  Since $G' = Z(G)$ is characteristic, $\Aut G \to \Psi\Isom(\Bi (G))$ by $\phi\mapsto (\phi|_{G/Z(G)};\phi|_{G'})$.   The kernel is $C_{\Aut G}(G/Z(G))\cong \hom(G/Z(G),G')$ acting as described in (ii).  Compare \cite[Propositions 3.8]{Wilson:unique-cent}.
\end{proof}

\begin{remark}
A detour into abstraction explains a few subtle choices in our definitions.
Baer's design for $\Bi$ is more clever than our treatment in that the role of $Z(G)$ can be replaced with a normal subgroup $M$ between $G'$ and $Z(G)$.  This allows one to insist that $M$ be fully invariant, perhaps even $G'$.  That choice makes $G \mapsto \Bi_M(G)$ a functor from the category of nilpotent groups of class at most $2$ to the category of alternating bimaps equipped with an appropriate set of morphisms.  However, such bimaps can be degenerate.  Instead, our choice of $M = Z(G)$ establishes a functor from the category of nilpotent groups of class at most $2$ equipped with isoclinisms into the category of nondegenerate alternating bimaps equipped with $\mathbb{Z}$-pseudo-isometries.
\end{remark}

%-------
\section{Recognizing quotients of Heisenberg groups}\label{sec:rec}

In this section, we focus on determining when a group $G$ is an epimorphic image of a generalized Heisenberg group $H_m(K)$.  To be clear, we do not mean that $G$ should be specified by matrices over the field $K$, in fact, both $K$ and $m$ are not known at the start and instead the abstract group properties of $G$ must be used to reconstruct $K$ and $m$.   This is necessary since we might only know a set of generators as permutations or matrices for an arbitrary representation of $G$, or a polycyclic presentation of $G$.  In such instances, $K$ and $m$ are not provided.  Indeed, one may even ask if the field $K$ is necessary to define a generalized Heisenberg group, which we affirm by proving that one may always recover an isomorphic copy of $K$ from the multiplication of a generalized Heisenberg group (\thmref{thm:rec-Hei}).  Therefore, the representation of the group is irrelevant.  We then generalize this technique to recognize abstract groups that are epimorphic images of generalized Heisenberg groups (\thmref{thm:rec-q-Hei}).  The tools used to recognize these groups lead directly to the proofs of our main theorems in the following section.

%----------
\subsection{Centroids}\label{sec:centroid}
A \emph{centroid}\footnote{This definition is the generalization of centroids of non-associative rings \cite[pp. 147--153]{Kaplansky:rings}. For bimaps this appears for the first time in \cite{Myasnikov} under the name \emph{enrichment ring}, and in this general
form in \cite[Section 5.2]{Wilson:direct-prod}.} of an alternating bimap $b:V \times V \to W$ is a ring $C$ over which $b$ is a $C$-bimap and $C$ is universal with that property.  That is to say, if $b$ is also an $R$-bimap, then there is a unique homomorphism $\varphi:R \to C$ such that for all $r\in R$, all $v\in V$, and all $w\in W$, $vr = v(r\varphi)$ and $wr = w(r\varphi)$.  As with non-associative algebras (cf. \cite[pp. 147--153]{Kaplansky:rings}), a centroid $C$ for $b$ always exists and it can be described as the ring:
\begin{align*}
    \Cent(b) & = \{(f;h)\in\End V \times \End W:
        \forall u,v\in V,
         b (uf,v) = b (u,v)h = b (u,vf)\}.
\end{align*}
The universal property of a centroid for $b$ makes it unique to $b$, up to a canonical isomorphism.

If $b$ is nondegenerate and $b (V,V) = W$, then $\Cent(b)$ is commutative:  for all $(f;h),(f';h')\in \Cent(b)$,
all $u\in U$, and all $v\in V$
\begin{align*}
	b (u(ff'),v) & = b (uf,vf') = b (u,vf')h = b (uf',v)h = b (u(f' f), v).
\end{align*}
As $b$ is nondegenerate, $f f' = f'f$.\footnote{The basic heuristic used here is  a \emph{three-pile-shuffle}: given three piles of cards (the three places for the functions), by moving one card from the top of one pile to the top of another eventually every possible permutation of the three piles can be had.  We argue similarly later without details.}
If $(f;h), (f';h')\in \Cent(b)$ and $h = h'$, then
\begin{align*}
    b (u(f-f'),v) & = b (uf,v) - b (uf',v) = b (u,v)h - b (u,v)h' = 0.
\end{align*}
Hence, $u (f - f') = 0$ for all $u \in U$ so that $f = f'$.  In a similar fashion, it follows that $\Cent(b)$ is faithfully represented in its restriction to $W$.
In particular, if $j:V \times V \to W$ is a nondegenerate $K$-bimap for a field $K$ with $\dim_K W = 1$ (i.e. a $K$-form), then $K$ embeds in $\Cent(j)$ and so $K\hookrightarrow \Cent(j)|_W\subseteq \End_K W\cong K$; thus, $\Cent(j)\cong K$.  For more on centroids of bimaps see \cite[Section 5.2]{Wilson:direct-prod}.

We use the centroid to recover $K$ from the multiplication of a generalized Heisenberg group $H$ over $K$.

\begin{thm}\label{thm:rec-Hei}  Let $H$ be a finite group with $1 = H^p < H' = Z(H) < H$.
Then $H$ is a generalized Heisenberg group if and only if $\Cent(\Bi (H))$ is a field and $Z(H)$ is $1$-dimensional over $\Cent(\Bi (H))$.
\end{thm}
\begin{proof}
For the forward direction, let $H$ be a generalized Heisenberg group.  By \exref{ex:j-Hei}, the map $\Bi (H) : V \times V \to W$, where $V = H/Z(H)$ and $W = H'$, is $\mathbb{Z}/p$-pseudo-isometric to a nondegenerate alternating $K$-form $j : K^{2m} \times K^{2m} \to K$, for some field $K$.  As above, $\Cent(\Bi (H))\cong K$, and as $W$ is $1$-dimensional over $K$, $W$ is also $1$-dimensional over $\Cent(\Bi (H))$.

Now, for the converse, suppose that $H$ is a finite group with $1 = H^p < H' = Z(H) < H$ and that $K := \Cent(\Bi (H))$ is a field with $H'$ a one-dimensional vector space over $K$.  By \propref{prop:Baer-correspondence}, $H$ is isomorphic to $\Grp(\Bi (H))$.  Our $\Bi (H)$ is a nondegenerate alternating $K$-form.  So, there is a $K$-pseudo-isometry $(\varphi;\hat{\varphi})$ from $\Bi (H)$ to $j:K^{2m} \times K^{2m} \to K$ as in \eqref{eq:alt} where $2m = \dim_K H/H'$ (\remref{rem:one-alt-form}).   Hence, $\Grp(\Bi (H))\cong \Grp(j)\cong H_m(K)$ (the final isomorphism from \propref{prop:Baer-correspondence} and \exref{ex:j-Hei}).  Therefore, $H$ is a generalized Heisenberg group over $K$.
\end{proof}

%-------
\subsection{Quotients of Heisenberg groups}\label{sec:q-Hei}
In this section, we focus on quotients of generalized Heisenberg groups $H$ and derive their initial properties.  Throughout this section, $H$ is a generalized
Heisenberg group.

\begin{lemma}\label{lem:Camina}
If $H$ is a generalized Heisenberg group, then
\begin{enumerate}[(i)]
\item for all $g \in H-H'$, $[g,H] = H'$ (equivalently $g^H = gH'$),
\item $H' = Z(H)$, and
\item For all $N\leq H$, $N\normaleq H$ if and only if $N \leq H'$ or $H'\leq N$.
\end{enumerate}
\end{lemma}
\begin{proof}
As in \exref{ex:j-Hei}, $H' = Z(H)$, $\Bi (H)$ is a nondegenerate alternating $K$-form, for the field $K = \Cent(\Bi (H))$,  and $H'$ is a $1$-dimensional $K$-vector space (\thmref{thm:rec-Hei}).  In particular, for each $g \in H - H'$, $u = Z(H)g$ is non-zero so $[g,H] = j(u,H/Z(H)) = K = H'$, so (i) holds.   Finally, for (iii) in the forward direction, if $g\in N - H'$, then $H'\leq [g,H]\leq N$.  For the converse, observe that $H' = Z(H)$, so all its subgroups are normal in $H$.  Likewise, all subgroups containing $H'$ are normal in $H$.
\end{proof}

Groups with the property of \lemref{lem:Camina}(i) are called {\em Camina groups}.  Note that all Camina groups of nilpotence class $2$ satisfy conditions (ii) and (iii).  These groups have many strong properties some of which contribute to the similarities between the many quotients of $H = H_m(K)$, and so, we return to this point of view in Section \ref{sec:invariants}.  For now, we simply note that the quotients of $H$ by normal subgroups containing $H'$ are elementary abelian and so unremarkable.  Thus, we only consider the remaining normal subgroups -- those properly contained in $H'$.

Fix a nonabelian group $G$ of class $2$ and an epimorphism $\phi:H \to G$.  First, we obtain alternating bimaps $j' = \Bi (H)$ and $b = \Bi (G)$.  As $G$ is nonabelian, by \lemref{lem:Camina}, $\ker \phi\leq H' = Z(H)$ and so $\phi$ factors through the natural $\mathbb{Z}/p$-linear isomorphism $\varphi : H/H' \to G/G'$ and also induces a $\mathbb{Z}/p$-linear epimorphism $\ct{\varphi}:H' \to G'$ where $\ker \ct{\varphi} = \ker\phi$.  It follows that $b(u\varphi,v\varphi)=j'(u,v)\ct{\varphi}$.  Indeed, $\varphi$ is invertible so we induce an alternating nondegenerate $K$-form $j:G/Z(G) \times G/Z(G) \to K$ by assigning $j(u,v) = j'(u\varphi^{-1}, v\varphi^{-1})$.  We observe that $b = j\ct{\varphi}$.  Thus, we have translated from epimorphisms of generalized Heisenberg groups over $K$ to alternating $\mathbb{Z}/p$-bimaps that factor through nondegenerate alternating $K$-forms.

We can also reverse the above translation as follows.  If $j : V \times V \to K$ is a nondegenerate alternating $K$-form on a $K$-vector space $V$ and $\pi:K\to W\neq 0$ is an epimorphism, then $(v,s) \mapsto (v,s\pi)$ is a group epimorphism from $\Grp (j)$ to $\Grp (j\pi)$.  Notice $H=\Grp(j)$ is generalized Heisenberg group and $\Grp(j\pi)$ is an epimorphic image of $H$.

We conclude that to study epimorphic images of a generalized Heisenberg group it suffices to study the $\mathbb{Z}/p$-bimap $j\pi$.  To study such bimaps, we introduce the ring of adjoints.

%-----
\subsection{Adjoints}
For a ring $R$, an \emph{$R$-mid-linear bimap} is a bimap $b:U \times  V \to W$ where $U$ is a right $R$-module, $V$ is a left $R$-module, and $b$ factors through the  $R$-tensor product $\otimes_R:U \times V \to U\otimes_R V$.  An \emph{adjoint} ring of a bimap $b:U \times V \to W$ is a ring $A$ over which $b$ is $A$-mid-linear and $A$ is universal with that property.  That is, whenever $b$ is $R$-mid-linear for some $R$, there is a unique homomorphism $\varphi:R \to A$ such that for all $r\in R$, all $u\in U$, and all $v\in V$, $ur = u(r\varphi)$ and $rv = (r\varphi)v$.  As with centroids (cf. Section \ref{sec:centroid}), an adjoint ring $A$ for $b$ exists and, up to a unique isomorphism, we may assume $A$ is:
\begin{align*}
    \Adj (b) = \{ (f,g)\in \End U \times (\End V)^{op} : \forall u\in U,\forall v\in V,~
        b (uf,v) = b (u,vg)~\}.
\end{align*}
In general, if $A \subseteq \End_K U \times (\End_K V)^{op}$, then $U$ is the right $A$-module and $V$ is a left $A$-module by assigning the actions: for all $(a,a') \in A$,  $u (a,a') = ua$, for all $u\in U$; and $(a,a')v = va'$, for all $v \in V$ (where we implicitly involve the property that composition in $\End_K V$ is as $(ab)^{op} = b^{op} a^{op}$, for $a,b \in \End_k V$).  So indeed, we are able to form $U \otimes_{\Adj (b)} V$ from the above definition.  The universal property follows immediately.

Adjoint rings in this generality seem to have appeared first in the study of central products \cite[Section 4]{Wilson:unique-cent}, and we will return to those implications in Section \ref{sec:invariants}.
\begin{ex}\label{ex:adj-j}
Let $K$ be a field.
If $j:K^{2m} \times K^{2m} \to K$ is the nondegenerate alternating $K$-form in \eqref{eq:alt}, then
\begin{align}
\label{eq:adj-j}
    \Adj (j) & = \left\{\left(\begin{bmatrix} A & B \\ C & D \end{bmatrix},
        \begin{bmatrix} D^t & -B^t \\ -C^t & A^t \end{bmatrix}\right): A,B,C,D\in M_m(K)\right\}.
\end{align}
\end{ex}

We have two important actions by $\GL_K(V)$. First, for each $x\in \GL_K(V)$ and each $\sum_i u_i \otimes v_i \in V \otimes_K V$, 
\begin{align*}
	\left( \sum_i u_i \otimes v_i \right)^x & = \sum_i (u_i x \otimes v_i x).
\end{align*}
Second, for each $x\in \GL_K(V)$ and each $(a,a') \in \End_K V \times (\End_K V)^{op}$, 
\begin{align*}
	(a,a')^x & = (x^{-1} a x, x^{-1} a'x) = (a^x, (a')^x).
\end{align*}
Hence, if $A \subseteq \End_K V \times (\End_K V)^{op}$ then $(V \otimes_A V)^x = V\otimes_{A^x} V$ and
$x$ induces a $K$-pseudo-isometry $(x;\ct{x})$ from $\otimes_A$ to $\otimes_{A^x}$.

Suppose $b$ is nondegenerate.  For all pairs $(f,g), (f',g') \in \Adj (b)$, if either $f = f'$ or $g = g'$ then $(f,g) = (f',g')$.  Thus, the projection $\Adj (b)|_U$ of $\Adj (b) \subseteq \End U \times (\End V)^{op}$ to $\End U$ is faithful.  As defined, the adjoint ring appears to involve $\mathbb{Z}$-linear endomorphisms.  However, a three-pile-shuffle shows that if $b$ is a $K$-bimap and $(f,g)\in \Adj (b)$ then both $f$ and $g$ are $K$-linear.  Hence, as $b$ is nondegenerate, $\Adj (b)|_U \subseteq \End_{\Cent (b)} U$.  Observe $\Cent (b)$ embeds in the center of $\Adj (b)$, again argued by a three-pile-shuffle; however, there are instances where the center of $\Adj (b)$ is larger than the image of $\Cent (b)$.

When $b$ is alternating, we must have $U = V$, and for every $(f,g) \in \Adj (b)$, it follows that $(g,f) \in \Adj (b)$.  More generally, we say  $b : V \times V \to W$ is \emph{Hermitian} if there is a $\theta \in \GL_{\mathbb{Z}/p} (W)$ such that for all $u, v \in V$, $b (v,u) = b (u,v)\theta$.  When $b$ is Hermitian, we see that $(f,g) \in \Adj (b)$ if and only if $(g,f)\in \Adj (b)$. Hence, $*:(f,g) \mapsto (g,f)$ is an anti-isomorphism of order at most $2$ on $\Adj (b)$; that is, it is an \emph{involution}.  When $b$ is nondegenerate and Hermitian, an involution is induced on $\Adj (b)|_V$, and we denote this involution by $f \mapsto f^*$ where $(f,f^*) \in \Adj (b)$.  For further details, see \cite[Section 3]{Wilson:unique-cent}.  We shall need the generality of Hermitian bimaps only long enough to prove that in our context every bimap we rely on remains alternating.

In general, for a ring $A$ if $V$ is a right $A$-module and $*$ is an involution on $A$, then we may treat $V$ also as a left $A$-module under the action $av := va^*$. For added clarity, we sometimes express this module by $V^*$.  Therefore, the map $\otimes_A : V \times V \to V \otimes_A V^*$ is defined.  Indeed, if $A = \Adj (b)|_V$ for a Hermitian bimap $b:V \times V \to W$, then $V \otimes_{\Adj (b)}V$ (as explained by the definition of $\Adj (b)$) is nothing other than $V \otimes_A V^*$.

First, we cite the following classic fact; cf. \cite[IX.10-11]{Jacobson} or \cite[Section 5.2]{Wilson:algo-cent}.

\begin{thm}\label{thm:*-simples}
Let $K$ be a finite field and $V$ a finite-dimensional $K$-vector space. If $A = \End_K V$ and $*$ is an involution on $A$, then there is a nondegenerate Hermitian $K$-form $d : V \times V \to K$ such that $A = \Adj (d)|_V$ with the involutions also equal.
\end{thm}

\thmref{thm:*-simples} allows us to invoke the classifications of nondegenerate Hermitian forms (which in our context includes alternating and symmetric forms as well as the typical Hermitian form).  That classification will be used to prove the next Theorem.

\begin{thm}\label{thm:simple-adj-tensor}
Let $K$ be a finite field and $V$ a finite-dimensional $K$-vector space.  If $A = \End_K V$ and $*$ is an involution on $A$, then $V \otimes_{A} V^* \cong K$; in particular, $\otimes_A : V \times V \to V \otimes_A V^*$ is a nondegenerate $K$-form.  Moreover, if $A$ is isomorphic to $\Adj (j)$ (as $*$-rings) for a nondegenerate alternating $K$-form $j : V \times V \to K$, then $j = \otimes_A \hat{\j}$ for a $K$-linear isomorphism $\hat {\j} : V \otimes_A V^* \to K$; indeed, $\otimes_A$ is an alternating nondegenerate $K$-form on $V$.
\end{thm}

Our proof of \thmref{thm:simple-adj-tensor} uses some vocabulary borrowed from \cite[Sections 3--4]{Wilson:unique-cent}.  Suppose that $b : V \times V \to W$ is a nondegenerate Hermitian $k$-bimap.  A $\perp$-decomposition is a $\oplus$-decomposition $V = X_1 \oplus \cdots\oplus X_s$ where none of the $X_i$ are trivial and for all $1 \leq i < j \leq s$, we have $b (X_i,X_j) = 0$ (which implies  $b (X_j,X_i) = 0$).  We denote this by $b = (b|_{X_1}) \perp \cdots \perp (b|_{X_s})$.  Observe that $b$ is conceptually an `orthogonal sum' in the following sense:
\begin{align}
    b (x_1 + \cdots +x_s, x'_1 + \cdots + x'_s) & = b (x_1, x'_1) + \cdots + b (x_s, x'_s)
\end{align}
where for each $i$ satisfying $1 \le i \le s$, we have $x_i,x'_i\in X_i$.  A Hermitian bimap $b$ is \emph{$\perp$-indecomposable} if it has exactly one $\perp$-decomposition.  A $\perp$-decomposition is \emph{fully refined} if its constituents are $\perp$-indecomposable.

\begin{ex}\label{ex:form-decomp}
For a finite field $K$, every nondegenerate Hermitian $K$-form $d$ has a fully refined $\perp$-decomposition into \emph{hyperbolic lines} $\langle e,f\rangle$ (where $d(e,e)=0=d(f,f)$ and $d(e,f)=1$), and \emph{anisotropic} points $\langle u\rangle$ (where $d(u,u)\neq 0$).
\end{ex}

\begin{lemma}\label{lem:ext}
Let $K$ be a finite field and $V$ and $W$ two $K$-vector spaces.  If $d : V \times V \to W$ is a $\perp$-indecomposable nondegenerate Hermitian $K$-form, then $\dim_K (V \otimes_{\Adj (d)} V)=1$.  Furthermore, if $\dim V = 2$ then $\otimes_{\Adj (d)} : V \times V \to V \otimes_{\Adj (d)} V$ is an alternating nondegenerate form.
\end{lemma}

\begin{proof}
By \exref{ex:form-decomp}, $0 < \dim_K V \leq 2$.

If $V = Kv$ for some $0 \neq v \in V$ then $\End_K V = \{ (v \mapsto sv) : s \in K \}$. As $d$ is a nondegenerate $K$-form, $\Cent(d) \cong K$.  Also, $\End_k V = \Cent (d)|_V \subseteq \Adj (d)|_V \subseteq \End_K V$ so that $\Cent (d)|_V = \Adj (d)|_V$.  As $\Adj (d)$ is faithfully represented as $K$-endomorphisms on $V$, $\Adj (d) = \{ (v \mapsto sv, v \mapsto sv) : s \in K \}$.  It follows that $(\alpha v) \otimes (\beta v) \mapsto \alpha \beta$ determines an isomorphism $V \otimes_{\Adj (d)} V \cong K$ as $K$-vector spaces.

Now, let $\dim_K V = 2$; that is $V = \langle e,f \rangle$ where $d (e,e) = 0 = d (f,f)$ and $d (e,f) = 1$.  If $u \in V$ is such that $d (u,u) \neq 0$, then $\langle u \rangle \cap u^{\perp} = 0$, and so, $d$ has a $\perp$-decomposition $V = \langle u \rangle \oplus u^{\perp}$.  Yet, we are assuming that $d$ is $\perp$-indecomposable, and so, $d$ must be alternating. Hence, in the $e,f$ basis, $d (u,v) = u \begin{bmatrix} 0 & 1\\ -1  & 0 \end{bmatrix} v^t$.
As $\left(\begin{bmatrix}b & b \\ -a & -a \end{bmatrix},\begin{bmatrix} -a & -b \\ a & b \end{bmatrix}\right)\in \Adj (d) =: A$, for all $[a,b] \in K^2$:
\begin{align*}
    0\otimes_A 0 & =    [a,b]\begin{bmatrix} b & b \\ -a & -a \end{bmatrix} \otimes_A [0,1]
        = [a,b]\otimes_A [0,1]\begin{bmatrix} -a & -b \\ a & b\end{bmatrix}
        = [a,b]\otimes_A [a,b].
\end{align*}
Thus, $K \cong K^2 \wedge_K K^2 := K^2 \otimes_K K^2/\langle u \otimes u : u \in K^2 \rangle$ maps $K$-linearly onto $K^2 \otimes_{A} K^2$; so, $\dim_K (K^2 \otimes_A K^2)\leq 1$.  By the definition of the adjoint ring, $d$ factors through $K^2 \otimes_A K^2$ so there is a canonical non-trivial $K$-linear mapping $\hat {d}$ of $K^2 \otimes_A K^2$ into $K$.  By considering dimensions, we see that $\hat {d}$ is a $K$-linear isomorphism.
\end{proof}

Now, we translate these geometric notions into ring theory so that we may prove \thmref{thm:simple-adj-tensor}.  In a ring $A$ with involution $*$, we call an element $e\in A$ \emph{$*$-invariant} if $e^*=e$.  If $e^2=e\neq 0$, then we say $e$ is \emph{idempotent}.  Two idempotents $e,f\in A$ are \emph{orthogonal} if $ef=0=fe$.  We say $e$ is a \emph{$*$-invariant-primitive idempotent} if $e^*=e=e^2\neq 0$ and $e$ is not the sum of two orthogonal $*$-invariant idempotents (noting that in the convention of Curtis-Reiner we do not permit $0$ as an idempotent).  Every finite $*$-ring $A$ has a set $\mathcal{E}$ of pairwise orthogonal $*$-invariant-primitive idempotents that sum to $1$. See
\cite[Section 4]{Wilson:algo-cent}.

In general, $\perp$-decompositions are difficult to recognize for arbitrary bimaps, and the key tool is to describe these decompositions through the ring $\Adj (b)$.
When $A\subseteq \End V$ and $e\in A$ is an idempotent, we know that $V=Ve\oplus V(1-e)$.  Now, if $e\in \Adj (b)$ is a $*$-invariant idempotent, then $b (Ve,V(1-e))=b (V,V(1-e)e)=0$ so that $b=(b|_{Ve})\perp (b|_{V(1-e)})$.  This process can also be reversed.  These are the mechanics that underpin the following tool.

\begin{thm}\cite[Corollary 4.5]{Wilson:algo-cent}\label{thm:perp-idemp}
The fully refined $\perp$-decompositions of a nondegenerate Hermitian bimap $b$ are in one-to-one correspondence with the sets of pairwise orthogonal $*$-invariant-primitive idempotents of $\Adj (b)$ that sum to $1$.
\end{thm}

\begin{proof}[Proof of \thmref{thm:simple-adj-tensor}]
By \thmref{thm:*-simples}, there is a nondegenerate Hermitian $K$-form $d : V \times V \to K$, where $K$ is the center of $A$, such that $A = \Adj (d)|_V$ with associated involution.  Using \exref{ex:form-decomp}, we obtain a fully refined $\perp$-decomposition of $V$ into hyperbolic points and anisotropic points.  Using \thmref{thm:perp-idemp}, there is a set $\mathcal{E}=\{e_1,\dots,e_m\}\subseteq A$ of pairwise orthogonal $*$-invariant-primitive idempotents whose $1$-eigenspaces on $V$ are $1$- or $2$-dimensional over $K$ according to whether the associated $\perp$-factor is anisotropic or hyperbolic.  However, $d$ factors through $\otimes_A$ (as $A=\Adj (d)|_V$ as a $*$-ring), and so, $A\subseteq \Adj (\otimes_A)|_V\subseteq \Adj (d)|_V=A$.  Furthermore, the involutions also agree.  Applying \thmref{thm:perp-idemp} in the opposite direction to the nondegenerate Hermitian bimap $\otimes_A$, we find $\otimes_{A} = b_1 \perp \cdots \perp b_m$ where each $b_i = (\otimes_A)|_{Ve_i} = (\otimes_{e_i A e_i})|_{Ve_i}$ is a $\perp$-indecomposable nondegenerate Hermitian $K$-bimap.  Therefore, $\otimes_{A}$ is a $K$-form so long as $1=\dim_K b_i(Ve_i,Ve_i)=\dim_K (Ve_i\otimes_{e_iAe_i} (Ve_i)^*)$, for all $i$ in $\{1,\dots,m\}$.  Since $A=\Adj (d)|_V$, $e_i Ae_i=\Adj (d_i)|_{Ve_i}$, where $d_i=d|_{Ve_i}$.  By \lemref{lem:ext}, we see that $\dim_K (Ve_i\otimes_{e_i A e_i} (Ve_i)^*)=\dim_K (Ve_i\otimes_{\Adj (d_i)|_{Ve_i}}Ve_i)=1$.

Next, suppose that $\tau : A \cong \Adj (j)|_V$ is a $K$-linear $*$-ring isomorphism for an alternating nondegenerate $K$-form $j$ on $V$.  Observe that $A=\End_K V=\Adj (j)|_V$ as rings so that $\tau$ is a $K$-linear ring automorphism of $\End_K V$.  From the Skolem-Noether theorem \cite[IX.10-11]{Jacobson}, $\tau$ is an inner automorphism, and so, there is an endomorphism $x\in \End_K V$ such that for all $a\in A$, $a\tau=x^{-1} ax$.  As $\tau$ is $*$-preserving, it follows that $\Adj (d)=\{(a,a^*):a\in A\}^x=\Adj (j)$ and therefore, $\otimes_{A}=\otimes_{\Adj (\otimes_A)}$ is pseudo-isometric to $\otimes_{\Adj (j)}$.  The latter is $K$-pseudo-isometric to $j$.  In particular, $\otimes_A$ is an alternating nondegenerate $K$-form.
\end{proof}

\thmref{thm:simple-adj-tensor} allows us to recognize quotients of Heisenberg groups.  Recall from the end of Section \ref{sec:q-Hei} that our interest is to recognize nondegenerate $\mathbb{Z}/p$-bimaps that factor through an alternating nondegenerate $K$-form $j$.

\begin{coro}\label{coro:quotient-Adj-1}
Let $K$ be a finite field, $V$ a $K$-vector space, and $W\neq 0$ a $\mathbb{Z}/p$-vector space.  If $j : V \times V \to K$ is a nondegenerate alternating $K$-form and $\pi : K \to W$ is a $\mathbb{Z}/p$-linear epimorphism, then $j\pi$ is alternating and nondegenerate, $\Adj (j\pi)$ is simple and acts irreducibly on $V$, and $\otimes_{\Adj (j\pi)}$ is an alternating nondegenerate $k$-form where $k$ is a subfield of $K$ isomorphic to the center of $\Adj (j\pi)$.
\end{coro}

We stress that \corref{coro:quotient-Adj-1} does not insist the $k$ is $K$.  For example, a $\mathbb {Z}/p$-linear epimorphism $\pi : K \to \mathbb {Z}/p$ will have $\Adj (j\pi)|_V \cong M_{2me} (\mathbb{Z}/p)$ where $e = [K:\mathbb{Z}/p]$, so it is not possible in general to assume $k=K$.

\begin{proof}
Suppose for some $0 \neq u \in V$, that for all $v \in V$ we have $j(u,v)\pi = 0$.  As $j$ is nondegenerate, there is an element $v \in V$ such that $j (u,v) =: s \neq 0$.  Now, for all $t \in K$, $t\pi = j (u,ts^{-1} v)\pi=0$, so $K\pi = 0$.  This is excluded by the assumptions on $\pi$.  Hence, $j\pi$ is nondegenerate.

Next, observe that $(f,f^*) \in \Adj (j)$ implies that for all $u,v \in V$, $j (uf,v) = j (u,vf^*)$, and so, also $j (uf,v) \pi=j (u,vf^*) \pi$, showing that $(f,f^*) \in \Adj (j\pi)$.  It follows that $\Adj (j)$ is contained in $\Adj (j\pi)$ as a $*$-subring.  As both $j$ and $j\pi$ are nondegenerate, $\Adj (j)|_V$ and $\Adj (j\pi)|_V$ are faithful representations on $V$ and $\Adj (j)|_V \subseteq \Adj (j\pi)|_V$ with the involution on $\Adj (j)|_V$ the restriction of the involution on $\Adj (j\pi)|_V$.  Because $j$ is a nondegenerate $K$-form, we have as rings $\End_K V=\Adj (j)|_V$ (cf. \exref{ex:adj-j}), and so, as rings
\begin{align*}
    \End_K V=\Adj (j)|_V \subseteq \Adj (j\pi)|_V \subseteq \End_{\mathbb{Z}/p} V.
\end{align*}
Because $V$ is a simple $\Adj (j)$-module, it is also a simple $\Adj (j\pi)$-module; in particular, as a ring $\Adj (j\pi)|_V$ is a simple subring of $\End_{\mathbb{Z}/p} V$ (i.e. $\Adj (j\pi)|_V$ is a finite primitive ring so it is simple).  Also, $\Adj (j\pi)$ contains a copy of $K$ (as scalar multiplication in $\End_K V$), the center $k$ of $\Adj (j\pi)$ is a subfield of this copy of $K$.

Every finite simple ring $R$ is isomorphic to the ring of endomorphisms of a finite-dimensional vector space over the center of $R$.  So $\Adj (j\pi) \cong \End_k U$ where $k$ is the center of $\Adj (j\pi)$ and $U$ is a finite-dimensional $k$-vector space.  As such, $U$ is an irreducible $\Adj (j\pi)$-module, but finite simple rings have one isomorphism type of simple module and so $U\cong V$ as $\Adj (j\pi)$-modules.  In particular, $\Adj (j\pi)\cong \End_k U\cong \End_k V$.  Since $\Adj (j\pi)|_{V}$ is a faithful representation of in $\End_k V$, it follows that $\Adj (j\pi)|_V=\End_k V$. The hypotheses of \thmref{thm:simple-adj-tensor} are now satisfied
by $\Adj (j\pi)|_V$, and so, $\otimes_{\Adj (j\pi)}$ is a nondegenerate $k$-form.

Finally, we must show that $\otimes_{\Adj (j\pi)}$ is alternating.  As
$\Adj (j)\subseteq \Adj (j\pi)$, $\otimes_{\Adj (j\pi)}$ factors
through $\otimes_{\Adj (j)}$.  By the final implication of
\thmref{thm:simple-adj-tensor}, $\otimes_{\Adj (j)}$ is alternating.
Therefore, $\otimes_{\Adj (j\pi)}$ is alternating as well.
\end{proof}

\begin{remark}\label{rem:tensor}
The usual technique for studying the alternating $\mathbb{Z}/p$-bimaps $b:V \times V \to W$ on $V=K^{2m}$ is to pull back to the $\mathbb{Z}/p$-exterior square $\wedge:V \times  V \to V\wedge_{\mathbb{Z}/p} V$.  However, $\dim_{\mathbb{Z}/p} (V\wedge V)\in \Theta(m^2\dim^2_{\mathbb{Z}/p} K)$.  In our context, $\dim_{\mathbb{Z}/p} W\leq \dim_{\mathbb{Z}/p} K$, and so, we have a very large gap between $\dim V\wedge_{\mathbb{Z}/p} V$ and $\dim_{\mathbb{Z}/p} W$.  Using $\otimes_{\Adj (b)}$ allows us to pull back (in a canonical way) to an alternating $\mathbb{Z}/p$-bimap $V \times V \to V\otimes_{\Adj (b)} V$, where $\dim_{\mathbb{Z}/p} V\otimes_{\Adj (b)}V \leq \dim_{\mathbb{Z}/p} K$.  
\end{remark}

%----
\subsection{Recognizing quotients of Heisenberg groups}
Interpreting \corref{coro:quotient-Adj-1} for generalized Heisenberg groups makes for a simple and computable test for when a group is isomorphic to a quotient of an odd order generalized Heisenberg group.

\begin{thm}\label{thm:rec-q-Hei}
Fix a group $G$ with $1=G^p < G'=Z(G) < G$, and a generalized Heisenberg group $H_{\ell}(K)$.  The following are equivalent.
\begin{enumerate}[(i)]
\item $G$ is an epimorphic image of $H_{\ell}(K)$.
\item $\Adj (\Bi (G))$ acts irreducibly on $G/Z(G)$ and is $*$-isomorphic to $\Adj (j)$ for a nondegenerate alternating $k$-form $j$ on $G/Z(G)$, for a subfield $k$ of $K$ isomorphic to the center of $\Adj (\Bi (G))$.
\end{enumerate}
\end{thm}

\begin{proof}
Let $\phi : H_{\ell}(K) \to G$ be an epimorphism.  As discussed at the close of Section \ref{sec:q-Hei}, if we set $V = G/Z(G)$, $W = G'$ and $b = \Bi (G)$, then there is an alternating nondegenerate $K$-form $j : V \times V \to K$ induced from $H_{\ell}(K)$, and a $\mathbb{Z}/p$-linear epimorphism $\ct{\varphi} : K \to W$, such that $b = j\ct{\varphi}$.  Thus, $\Adj (b) = \Adj (j\pi) = \Adj (\otimes_{\Adj (j\pi)})$ (with equality as $*$-rings).  By
\corref{coro:quotient-Adj-1}, $\otimes_{\Adj (j\pi)}$ is a nondegenerate alternating $k$-form (possibly different from $j$) where $k$ is a subfield of $K$ and isomorphic to the
center of $\Adj (b)$. Furthermore, \corref{coro:quotient-Adj-1} also shows $\Adj (b)$ acts irreducibly on $V$.
Since the $*$-isomorphism type and representation of $\Adj (b)$ is a $\mathbb{Z}/p$-pseudo-isometry invariant, it follows that the $*$-isomorphism type and representation of $\Adj (\Bi (G))$ is an isomorphism invariant of $G$.  This proves that (i) implies (ii).

Next, we show (ii) implies (i).  We assume that $A = \Adj (\Bi (G))$ acts (faithfully) irreducibly on $V = G/Z(G)$, so that $A|_V = \End_k V$ for a field $k$ isomorphic to the center of $A$.  Furthermore, $A$ is $*$-isomorphic to $\Adj (j)$ for a nondegenerate alternating $F$-form $j$ on $V$, for some subfield $F$ of $K$.  The involution on $\Adj (j)$ (and therefore on $A$) preserves the center (cf. \exref{ex:adj-j}), and so, the isomorphism $A \to \Adj (j)$ induces an isomorphism $k \cong F$.  Therefore, we treat $j$ as an alternating $k$-form, and $\Adj (j)|_V = \End_k V = A|_V$.  We now apply \thmref{thm:simple-adj-tensor}, and we find that $j' := \otimes_{A}$ is an alternating nondegenerate $k$-form.  This implies that $H := \Grp (j')$ is a generalized Heisenberg group (cf. \exref{ex:j-Hei}).  By the universal properties of tensors, $\Bi (G) = j'\pi$ for a (unique) additive map
$\pi : V \otimes_{A} V^* \to G'$.  Letting $N=\ker \pi$, we have $\Grp (j'\pi) \cong H/N$.  Finally, by \propref{prop:Baer-correspondence}, we know that $G \cong \Grp (\Bi (G)) = \Grp (j'\pi) \cong H/N$.  Therefore, $G$ is an epimorphic image of a generalized Heisenberg group.
\end{proof}

\begin{remark}
We can also view \thmref{thm:rec-q-Hei} as stating that $G$, as in \thmref{thm:rec-q-Hei}, is an epimorphic image of $H_{\ell} (K)$ if and only if $\otimes_{\Adj (\Bi (G))}$ is an alternating nondegenerate $k$-form for subfield $k$ of $K$ such that $\dim_k G/Z(G)=2\ell\cdot [K:k]$.  This follows by translating condition (ii) using \corref{coro:quotient-Adj-1} and considering the associated requirements on dimensions.
\end{remark}

\subsection{Indigenous quotients}
An implication of \thmref{thm:rec-q-Hei} is that every nonabelian quotient $H/N$ of a generalized Heisenberg group implicitly determines a smallest generalized Heisenberg group
of which it is a quotient.  Specifically, if $K$ is the center of $\Adj (\Bi (H/N))$ and $(H/N)/(H/N)'$ is $2m$ dimensional over $K$, then we write:
\begin{align}\label{eq:floor}
    \lfloor H/N \rfloor & = H_m(K).
\end{align}
In the language of our introduction, we say $H/N$ is \emph{indigenous} to $H$ if $H\cong \lfloor H/N\rfloor$.  As discussed in Section \ref{sec:q-Hei}, there is a natural $\mathbb{Z}/p$-isometry $\phi$ from $\Bi (H/N)$ to $\Bi (H) \pi$, for an appropriate epimorphism $\pi$.  Thus, $\Adj (\Bi (H)) \subseteq \Adj (\Bi (H) \pi) = \Adj (\Bi (H/N))^{\phi}$.  So we have proved:
\begin{prop}\label{prop:indig}
$H/N$ is indigenous to $H$ if and only if 
	$$\Adj (\Bi (H)) = \Adj (\Bi (H) \pi) = \Adj (\Bi (H/N))^{\phi}$$ 
(where equality includes as rings with involution).
\end{prop}

There are many indigenous quotients, but to guarantee that all quotients of a certain size are indigenous to a Heisenberg group, we use some elementary number theory.

\begin{lemma}\label{lem:good-pair}
For every integer $n \geq 12$, there is an integer $d=d_n$ such that
\begin{enumerate}[(i)]
\item $2d + 2 \le n \le 3d$,
\item for all $i$ such that $n - 2d \leq i < d$, $i \ndivides d$, and
\item $d - \frac {5}{12} n \in O(1)$ (as functions of $n$).
\end{enumerate}
\end{lemma}

Note, \lemref{lem:good-pair}(ii) is satisfied whenever $d$ is prime.

\begin{proof}
Suppose first that $n \ge 60$, write $n = 12q + r$ for an integer $0 \leq r < 12$.  Note that $q \ge 5$.  Set $d = 5q + e$ where $e$ is an integer chosen between $1$ and $4$ so that $d$ is congruent modulo $30$ to one of $1$, $7$, $11$, $17$, $23$, or $29$.  Immediately (iii) follows.  Observe that $2d + 2 = 10 q + 2e + 2 \le 10 q + 10 < 12 q \le n$ since $q \ge 5$ and so $10 \le 2q$.  Also, $3d = 15q + 3e > 15q > 12q + r$ since $3q \ge 12 > r$.  Observe that $n - 2d = 12q + r - 2(5q + e) = 2q + r - 2e \ge 2q - 2e$.  Notice that $2e \le 8$, so if $q \ge 8$, then $n - 2d \ge q$.  Let $p$ be the smallest prime dividing $d$, and note that $p > 6$.  We have $d/p < d/6 = 5/6
q + e/6 \le 5/6 q + 4/6 < 5/6 q + 1/6 q = q$.  It follows that for all $i$ if $n-2d \leq i < d$, then $d/p < i$, and $d < ip$.  On the other hand, if $i$ divides $d$, then $d/i \ge p$, and so, $d \ge ip$. This is a contradiction, so $i \ndivides d$.  If $q = 5$, then $d = 29$, if $q = 6$, then $d = 31$, and if $q = 7$, then $d = 37$.  In each of these cases, $d$ is prime, and since $n - 2d \ge 2$, $(n,d)$ satisfies (ii).

For $12 \le n \le 15$, take $d = 5$.  For $16 \le n \le 21$, we take $d = 7$.  For $22 \le n \le 23$, take $d = 8$.  For $24 \le n \le 33$, take $d = 11$.  For $34 \le n \le 39$, take $d = 13$.  For $40 \le n \le 57$, take $d = 19$.  For $n = 58$ or $n = 59$, take $d = 23$.  One can check by hand that each of these pairs $(n,d)$ satisfy (i) and (ii).
\end{proof}

First, we show \lemref{lem:good-pair} (i) and (ii) guarantees that indigenous quotients exist.  Later, we will use part (iii) to show that indigenous quotients are plentiful.

\begin{prop}\label{prop:stable-q}
Let $(n,d)$ be a pair as in \lemref{lem:good-pair} parts (i) and (ii).  If $H$ is a Heisenberg group of order $p^{3d}$ and $N\leq H'$ with $[H:N]=p^n$, then $H\cong \lfloor H/N\rfloor$.
\end{prop}
\begin{proof}
Let $b=\Bi (H/N):V \times V \to W$.  By \corref{coro:quotient-Adj-1} and \eqref{eq:adj-j}, the ring $\Adj (\Bi (H/N))$ is isomorphic as a ring to $M_{2m} (F)$ for a subfield $F$ of $K$ and where $K^2 \cong V \cong F^{2m}$.  Furthermore, $H_2= \lfloor H/N \rfloor$ is a generalized Heisenberg group over $F$ of degree $m$; hence, define $f$ by $|H_2'| = |F| = p^f$.  Let $H_2/M \cong H/N$ (and such an $M$ exists as $H_2 = \lfloor H/N \rfloor$).  It follows that
\begin{align}
    p^{n-2d} & = [H':N] = [H_2':M] = p^{n-2mf}.
\end{align}
Thus, $d = mf$, and furthermore, $n - 2d \leq f \leq d$ since $[H_2':M] \leq p^f$.  By the assumptions that $(n,d)$ satisfies \lemref{lem:good-pair} (ii) and $f \divides d$, it follows that $f = d$.  Thus, $F = K$ and $m = 1$.  So $H_2 \cong H$.
\end{proof}

%-------
\section{Proof of main theorems}\label{sec:main}

In this section we prove Theorems \ref{thm:main} and \ref{algo:main}.

%----------
\subsection{Lifting isomorphisms}
We begin with an observation which is likely well-known.

\begin{thm}\label{thm:aut-Hei}
If $H$ is a generalized Heisenberg group of degree $m$ over $K$ of characteristic $p$, then $\Aut H  = \Psi\Isom(\Bi (H))\ltimes \hom_{\mathbb{Z}/p}(K^{2m},K))$ and
\begin{align*}
    \Psi\Isom(\Bi (H)) & = \Gal(K)\ltimes (K^{\times}\ltimes \Sp(2m,K)).
\end{align*}
(Note that $\hom_{\mathbb{Z}/p}(K^{2m},K)$ corresponds to the inner automorphisms of $H$.)
\end{thm}

\begin{proof}
The structure of $\Aut H$ is explained by \propref{prop:Baer-correspondence}; so we concentrate on $\Psi\Isom(\Bi (H))=\Psi\Isom(j)$ where $j:K^{2m} \times K^{2m} \to K$ a nondegenerate alternating $K$-form.

First, for all $(\phi;\ct{\phi}) \in \Psi\Isom(j)$, and all $(f,g) \in
\Adj (j)$,  $(f,g)^{(\phi,\ct{\phi})} := (f^{\phi},
g^{\phi})\in \Adj (j)$. Therefore, $\Psi\Isom(j)$ acts on the center
$K$ of $\Adj (j)$ as a group of ring automorphisms.  The action on the center induces a group homomorphism $\Psi\Isom(j) \to \Gal(K)$
denoted $s \mapsto s^{\phi}$. In particular, if $s \in K$ and $u \in
K^{2m}$, then
\begin{align*}
    (su)\phi & =u(s I_{2m}\cdot \phi)=u(\phi\cdot s^{\phi} I_{2m})=s^{\phi} (u\phi).
\end{align*}
In particular, elements of $\Psi\Isom(j)|_V$ are $K$-semilinear. Let $u,v\in V$ be such that
$j(u,v)=s\neq 0$.  For each $t\in K$, $t=j(u,ts^{-1} v)$ and so
\begin{align}
t\ct{\phi} & =j(u\phi, (ts^{-1} v)\phi)= j(u\phi, t^{\phi} (s^{-1}v\phi))= t^{\phi} j(u\phi, (s^{-1} v)\phi)
	= t^{\phi} (1\ct{\phi}).
\end{align}
In particular, $t\ct{\phi} = t^{\phi} \lambda$, where $\lambda_{\phi} := 1\ct{\phi}\in K^{\times}$ proving $\ct{\phi}\in \Gal(K)\ltimes K^{\times}\leq \GL_{\mathbb{Z}/p}(K)$.
Consequently,  $(\phi; \ct{\phi})\mapsto \ct{\phi}$ is a group homomorphism $\Psi\Isom(j) \to \Gal(K)\ltimes K^{\times}$ with kernel $\Isom(j)$.

Now, for each $\tau\in \Gal(K)$, $(v\mapsto v^{\tau}; s\mapsto s^\tau)\in
\Psi\Isom(j)$; hence, $\Gal(K)\hookrightarrow \Psi\Isom(j)$ and its
image splits with the $K$-linear pseudo-isometries $\Psi\Isom_K(j)$.
The group $\Psi\Isom_K(j)$ admits $\Isom(j)=\Sp(2m,K)$ as well as
$K^{\times}$ since
\begin{align*}
    \left(v\mapsto v \begin{bmatrix} I_m & 0 \\ 0 & sI_m\end{bmatrix}; \alpha\mapsto s\alpha\right)&\in \Psi\Isom_K(j)
        & (\forall s\in K^{\times}).
\end{align*}
The image of $K^{\times}$ splits with $\Isom(j)$ in $\Psi\Isom_K(j)$.  This
completes the proof.
\end{proof}

We now turn to the question of lifting isomorphisms of quotients of $H$ to automorphisms of $H$.  Throughout this discussion, $K/(\mathbb{Z}/p)$ is a finite field extension and $b: V \times V \to W$ is a $\mathbb{Z}/p$-bimap.

Suppose that $\pi: W \to X\neq 0$ and $\tau: W \to Y\neq 0$ are $\mathbb{Z}/p$-linear epimorphisms.  Set $c=b\pi$ and $d=b\tau$.  If $(\phi; \ct{\phi}): c \to d$ is a $\mathbb{Z}/p$-pseudo-isometry, then $\Adj (c)^{\varphi}=\Adj (d)$ and so there is an isomorphism
$\ct{\Phi}:V\underset{\Adj (c)}{\otimes} V \to V\underset{\Adj (d)}{\otimes} V$ where
\begin{align}\label{eq:Phi}
    (u \otimes v) \ct{\Phi} & = u \phi \otimes v \phi & (\forall u, v\in V).
\end{align}
So $(\phi;\ct{\Phi})$ is a $\mathbb{Z}/p$-pseudo-isometry from $\otimes_{\Adj (c)}$ to $\otimes_{\Adj (d)}$.  Also, as $c$ factors through $\otimes_{\Adj (c)}$ there is an epimorphism $\hat{c}:V\otimes_{\Adj (c)} V \to X$  such that $c=\otimes_{\Adj (c)}\hat{c}$ and an isomorphism $\bar{c}:(V\otimes_{\Adj (c)} V)/(\ker \hat{c})\cong X$.  The same construction is applied to $d$.  Immediately, $\hat{c}\ct{\phi}=\ct{\Phi}\hat{d}$ and so $(\ker\hat{c})\ct{\Phi} = \ker\hat{d}$.  Hence,  $\ct{\Phi}$ induces an isomorphism $\gamma$ from $(V\otimes_{\Adj (c)} V)/(\ker \hat{c})$ to $(V\otimes_{\Adj (d)} V)/(\ker \hat{d})$ such that $\ct{\phi}=\bar{c}^{-1} \gamma \bar{d}$.   So in that sense, $\ct{\Phi}$ induces $\ct{\phi}$, and so, we say that $(\phi;\ct{\Phi})$ induces $(\phi;\ct{\phi})$. Finally, if $A := \Adj (c)=\Adj (d)$, then $(\phi;\ct{\Phi})$ is a $\mathbb{Z}/p$-pseudo-isometry of
$\otimes_{A}$ that induces the $\mathbb{Z}/p$-pseudo-isometry $(\phi;\ct{\phi})$.

\begin{thm}\label{thm:lift-iso}
Let $H$ be a generalized odd order Heisenberg group, and let $M$ and
$N$ be proper subgroups of $H'$. If $H/M$ and $H/N$ are indigenous quotients of $H$,
then every isomorphism $\varphi: H/M \to H/N$ is induced by an
automorphism $\Phi$ of $H$ with $M\Phi=N$.
\end{thm}
\begin{proof}
Choose $H=\Grp(j)$ for $j:K^{2m} \times K^{2m} \to K$ as in \eqref{eq:alt}, set $V=H/H'$, and fix the transversal $\ell:V \to K^{2m} \times 0\subseteq H$.  Treat $M,N < H'=0 \times  K$ as $\mathbb{Z}/p$-subspaces of $K$.  Let $\pi_M:K \to K/M$ and $\pi_N:K \to K/N$ be the natural projections.  There are also natural isomorphisms
	$$(H/M)/(H/M)'\overset{\tau_M}{\rightarrow} V \overset{\tau_N}{\leftarrow} (H/N)/(H/N)'.$$
We see that $(\tau_M; 1_{K/M})$ is an isometry from $\Bi (H/M)$ to $c=\Bi (H)\pi_M$ and $(\tau_N;1_{K/N})$ is an isometry from $\Bi (H/M)$ to $d=\Bi (H)\pi_N$.  Now, fix an isomorphism $\varphi: H/M \to H/N$ of groups. Set $\phi=\tau_N^{-1} (\varphi|_{(H/M)/(H/M)'})\tau_M$, which is a $\mathbb{Z}/p$-linear automorphism of $H/H'$.  Also, set $\ct{\phi}=\varphi|_{K/M}:K/M \to H/N$.  Thus, $(\phi;\ct{\phi})$ is a $\mathbb{Z}/p$-pseudo-isometry from $c$ to $d$.  Furthermore, $(\phi;\hat{\phi})$ induces an isomorphism $\Grp(\phi;\hat{\phi}):\Grp(\Bi (H/M)) \to \Grp(\Bi (H/N))$.  At this point we have constructed the outer square in the commutative diagram of \figref{fig:lifting} where the vertical isomorphisms are given by the Baer correspondence with respect to the fixed transversal $\ell$; cf. \propref{prop:Baer-correspondence}. 

Since we assume $H/M$ and $H/N$ are indigenous to $H$, $A=\Adj (c)=\Adj (\Bi (H))=\Adj (d)$. Therefore, \eqref{eq:Phi} determines a $\mathbb{Z}/p$-pseudo-isometry $(\phi;\ct{\Phi})$ of $\otimes_A$ that induces $(\phi;\ct{\phi})$.
By \corref{coro:quotient-Adj-1}, $\otimes_{A}$ is an alternating nondegenerate $K$-form, and this leads to a $\mathbb{Z}/p$-pseudo-isometry $(\tau;\ct{\tau})$ from $j$ to $\otimes_A$ (above).  We obtain $(\gamma;\ct{\gamma})=(\phi;\ct{\Phi})^{(\tau;\ct{\tau})}\in \Psi\Isom_{\mathbb{Z}/p}(j)$ and $(\gamma;\ct{\gamma})$ induces $(\phi;\ct{\phi})$.  Finally, $\Psi\Isom_{\mathbb{Z}/p}(j)$ embeds in $\Aut H$, and so, there is an automorphism $\Phi \in \Aut H$ such that $\Phi$ induces $(\gamma;\ct{\gamma})$, and so, it induces $\varphi$; in particular, $M\Phi=N$.  This describes the inner square in the diagram \figref{fig:lifting}.
\end{proof}
\begin{figure}
\begin{equation*}
\xymatrix{
H/M \ar[rrr]^{\varphi}\ar[ddd]^{\cong} & & & H/N\ar[ddd]^{\cong}\\
 & H\ar[d]^{\cong} \ar[r]^{\Phi}\ar[ul] & H\ar[ur]\ar[d]^{\cong} \\
 & \Grp(\otimes_A)\ar[r]_{\Grp(\phi;\ct{\Phi})}\ar[dl] & \Grp(\otimes_A)\ar[dr]\\
\Grp(c) \ar[rrr]_{\Grp(\phi;\ct{\phi})} & & & \Grp(d).
}
\end{equation*}
\caption{The diagram illustrating how to pass the isomorphism $\phi$ to the isomorphism $\Grp(\phi;\ct{\phi})$.  Then lift to the automorphism $\Grp(\phi;\ct{\Phi})$, and finally to the automorphism $\Phi$.}\label{fig:lifting}
\end{figure}
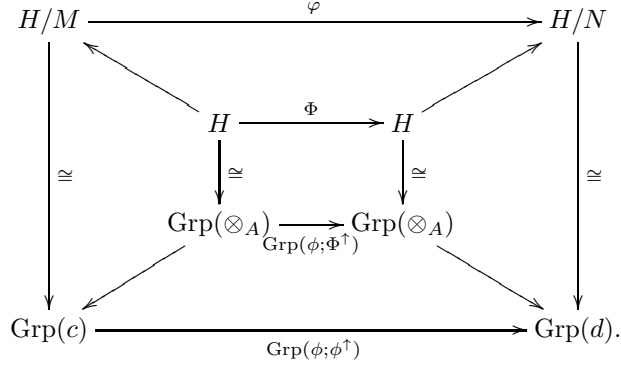
\begin{remark}
W. M. Kantor suggests that an alternative proof for \thmref{thm:lift-iso} might be obtained by considering the Schur multipliers.
\end{remark}

\thmref{thm:aut-Hei} implies the converse of \thmref{thm:lift-iso}  and so we have proved:

\begin{coro}\label{coro:orbits}
If $H$ is a generalized odd order Heisenberg group and $M,N < H'$ are such that $H/M$ and $H/N$ are indigenous to $H$, then $H/M\cong H/N$ if and only if there is an automorphism $\Phi$ of $H$ with $M\Phi=N$.  Thus, the isomorphism classes of the indigenous quotients of $H$ are in bijection with the $(\Aut H)$-orbits on the subgroups of $H'$.
\end{coro}

%----
\subsection{Proof of \thmref{thm:main}}
Let $(n,d)$ be a pair as in \lemref{lem:good-pair}, set $s=n-2d$, and fix $K$ to be a finite field of order $p^d$. Take $H$ to be a Heisenberg group over $K$, so $j=\Bi (H)$ is an alternating nondegenerate $K$-form on $K^{2}$.  Following \thmref{thm:aut-Hei}, $\Aut H$ maps onto $\Psi\Isom(j)$ and $\Aut H$ acts on the subgroups of $H'$ as $\Gal(K) \ltimes K^{\times}$ acts on the $\mathbb{Z}/p$-subspaces of $K$.  The number of subgroups of index $p^s$ in $H'$ is estimated by counting the number of $\mathbb{Z}/p$-subspaces of codimension $s$ in $K$ which is
\begin{align}
    \begin{bmatrix} d \\ s \end{bmatrix}_p  & = \prod_{i=1}^s \frac{p^d-p^{i-1}}{p^s-p^{i-1}} \geq p^{s(d-s)}.
\end{align}
The number of $(\Aut H)$-orbits on the subgroups $H'$ of index $p^s$ is bounded below by  $p^{s(d-s)}/(|\Gal(K)|(|K|-1))$.   By \propref{prop:stable-q}, quotients of size $p^{2d+s}=p^n$ are indigenous to $H$.  Hence, in light of \corref{coro:orbits}, the number of isomorphism classes of quotients of $H$ of order $p^{n}$ is at least:
\begin{align*}
    \frac{p^{s(d-s)} }{ d (p^{d}-1)}\geq p^{-s^2 +(s-1)d-\log_p d}.
\end{align*}
When we optimize $f(s,d)=-s^2+(s-1)d-\log_p d$ over $d$ subject to the constraint that $n=2d+s$, we find the maximum occurs for $d\in 5n/12+O(1)$ and $s\in n/6+O(1)$ and the number of orbits is at least $p^{n^2/24 + O(n)}$.  By \lemref{lem:good-pair}(iii) the pair $(n,d)$ attains this asymptotic maximum.  Therefore, the Heisenberg group of order $p^{3d}=p^{5n/4+O(1)}$ over a field of order $p^d$ has $p^{n^2/24+O(n)}$ pairwise nonisomorphic quotients of order $p^n$.\hfill $\Box$

%--------
\subsection{Proof of \thmref{algo:main}}
As we mentioned in the introduction, our original algorithm applied only to permutation groups, but using a result of L. Ronyai, we can extend these to more general settings.

\begin{proof}[Proof of \thmref{algo:main} (i)]
Using the standard polynomial-time algorithms (cataloged in \cite[pp. 4--6]{Seress:book} for permutation groups, in \cite{Luks:mat} for matrix groups, and in \cite{HoltEO} for polycyclic groups with a black-box multiplication), compute $G^p$, $Z(G)$, and $G'$, and then, certify that
$1 = G^p < G' = Z(G) < G$; otherwise, $G$ cannot be a nonabelian quotient of a
generalized Heisenberg group.  

Next, use the algorithms of
\cite[Section 5]{Wilson:algo-cent} to compute structure constants
for $b = \Bi (G)$, a basis for $\Adj (b)$, and recognize whether or
not $\Adj (b)$ is a simple ring acting irreducibly on $V = G/Z(G)$
and $*$-isomorphic to the adjoint ring of an alternating
nondegenerate form.  By \thmref{thm:rec-q-Hei}, at this point we have determined if $G$ is an epimorphic image of a generalized Heisenberg group.  

If $G$ is an epimorphic image of a generalized Heisenberg group then the algorithm creates
$\otimes_{\Adj (b)}$ along with the canonical projection $\pi : V
\otimes_{\Adj (b)} V  \to G'$.  Set $H = H_m (K)$ where $K$ is the center of $\Adj (b)$ and
$2m = \dim_K V$.  Finally, the algorithm computes a
standard hyperbolic basis for $\otimes_{\Adj (b)}$ and a change
of basis determines a pseudo-isometry $(\varphi;\ct{\varphi})$ from
$j = \Bi (H)$, $2m = \dim_K V$, to $\otimes_{\Adj (b)}$.  It follows that
$(\varphi;\ct{\varphi})$ induces an isomorphism $\Phi : H \to
\Grp (\otimes_{\Adj (b)})$ and $\pi$ determines an epimorphism
$\Gamma: \Grp (\otimes_{\Adj (b)}) \to G$ so that $\Phi\Gamma : H \to G$ is the desired epimorphism.  

The algorithms cited have both a deterministic version that runs in time polynomial in $\log |G|+p$, and non-deterministic version of the Las Vegas type with polynomial run time in $\log |G|$. In particular the algorithms are honest deterministic polynomial time algorithms for both permutation groups and for matrix groups in bounded characteristic.  This gives us the stated complexity of \thmref{algo:main}.  
\end{proof}

\begin{lemma}[Ronyai]\label{lem:Ronyai-trick}
Let $K/k$ be a finite extension of a finite field $k$.  There is a deterministic algorithm that given $k$-subspaces $U$ and $V$ of $K$, determines a $c\in K^{\times}$ such that $Uc = V$ or proves that no such $c$ exists.  The algorithm uses $O(\dim^6 K)$ operations in $k$.
\end{lemma}

\begin{proof}
First the algorithm decides if $\dim_k U = \dim_k V$, and if not, then it reports that $U$ and $V$ cannot be in the same $K^{\times}$-orbit.  Otherwise, the algorithm has $k$-bases $\{u_1,\dots,u_s\}$ and $\{v_1,\dots,v_s\}$ for $U$ and $V$ respectively.   If there exists a field element $c\in K$ such that $Uc=V$, then for each integer $1 \leq i \leq s$, there are field elements $\alpha_{i1}, \dots, \alpha_{ij}\in k$, such that
\begin{align}\label{eq:solutionset}
	u_i c & = v_1 \alpha_{i1} + \cdots + v_s \alpha_{is}.
\end{align}
Observe that these equations are $k$-linear in the variables $c$ and $\alpha_{i1},\dots,\alpha_{is}$.  To solve the system, we first fix a $k$-basis for $K$.  We then write $u_1, \dots, u_s$ and $v_s, \dots, v_s$ in this basis, and we write $c$ as linear combination in the basis for $K$ with unknown coefficient in $k$.  We then solve the equations determined by \eqref{eq:solutionset}.  This can be done with $O((\dim^2_k K)^3)$ operations in $k$.
\end{proof}

\begin{proof}[Proof of \thmref{algo:main}(ii)]
Using \thmref{algo:main} (i), we determine if the groups are indigenous quotients of a common Heisenberg group $H = H_m (K)$ for a finite field $K$ of size $p^d$.  This allows us to treat the input groups as quotients $H/M$ and $H/N$.  Furthermore, we determine if $[H:M] = [H:N]$, and if not, then the groups are nonisomorphic.

By \corref{coro:orbits}, the quotients $H/M$ and $H/N$ are isomorphic if and only if $N \in M^{\Aut H}$.  Because $\Aut H/C_{\Aut H} (H') \cong \Gal(K) \ltimes K^{\times}$, we fix a generator $\sigma$ for $\Gal(K)$.  Then, for each integer $1 \leq i\leq \dim_{\mathbb{Z}/p} K$, we use the algorithm of \lemref{lem:Ronyai-trick} to determine if there exists a field element $c\in K$, satisfying $(M\sigma^i)c=N$ (treating $M$ and $L$ as $\mathbb{Z}/p$-subspaces of $K$).  If this fails for each $i$ then $H/L$ is not isomorphic to $H/M$.  Otherwise, use the solution $(\sigma^i,c)\in \Gal(K)\ltimes K^{\times}$ to construct an automorphism $\Phi$ of $H$ where $M\Phi=N$ and so $\Phi$ induces an isomorphism $\phi=\Phi|_{H/M}:H/M\to H/N$.
\end{proof}

\begin{remark}\label{rem:nil-q-algo}
Our original proof used the observation that the size of $M^{\Aut H}$ is a divisor of $d(p^d-1)$.  The $(\Aut H)$-orbit of $M$ can be constructed from a basis for $M$ and $N$ can be tested for inclusion in $M^{\Aut H}$ by linear algebra at a cost of $O(d^3)$ for each of the $d(p^d-1)$ tests.  Hence, the total work is at worst $d^{4} p^d \in O(|H|^{1/(m+1)}\log^c |H|)$ for a constant $c$.  That was enough to obtain a polynomial bound on the algorithm's running time when the groups were specified by permutations.  (That uses the observation that nonabelian quotients of Heisenberg groups have permutation representations of degree at least $p^{2d}$.)

Our method still depends on exhausting over the elements in $\Gal(K)$, but this is dramatic decrease in the work required to list all of $\Gal(K)\ltimes K^{\times}$ (our original approach).  Both are substantial improvements over the traditional methods which would list all of $\Aut H$ in this context.  To see this we give a small survey of the standard methods some of which date back to work of Higman \cite[p. 10--12]{Higman:chic}.  

Higman defined a characteristic central series $\Phi^{(i)}$ for groups, now replaced by the lower exponent-$p$-central series. If $G$ and $J$ are $p$-groups and $G/\Phi^{(c)}(G)\cong J/\Phi^{(c)}(J)$, then there is a universal covering group $F$ mapping onto $G_{c+1} := G/\Phi^{(c+1)}(G)$ and $J_{c+1} := J/\Phi^{(c+1)}(J)$.   Thus, $G_{c+1}$ and $J_{c+1}$ are isomorphic if and only if their kernels in $F$ are in the same $(\Aut F)$-orbit.  Algorithms of this sort are collectively called \emph{nilpotent quotient algorithms} and have had many practical advances; for a survey see \cite{OBrien:iso}.  Yet, for $p$-groups of nilpotence class $2$ and order $N=p^n$, the universal covering groups $F$ in use can have order $p^{n+\binom{n}{2}}=N^{\log_{c} N+O(1)}$, $c$ depending on $p$, and the size of the $(\Aut F)$-orbits can reach $N^{\log_{c'} N+O(1)}$, $c'$ depending on $p$.  Indeed, for  quotients of order $N=p^n$ of a Heisenberg group of order $p^{5n/4+O(1)}$, the size of the orbits required by the general nilpotent quotient algorithms is:
\begin{align*}
    \frac{[\Aut F:C_{\Aut F}(F/F')|}{[\Aut H:C_{\Aut H}(H/H')]}
        & \approx \frac{|\GL(5n/6,p)|}{\frac{5n}{12}\cdot p^{5n/12} |\Sp(2,p^{5n/12})|}
         \in p^{\Theta(n^2)}=N^{\log_{d} N+\Theta(1)},
\end{align*}
where $d$ depends only on $p$.
The aspect of \thmref{algo:main} that permits a polynomial-time algorithm is summarized in \remref{rem:tensor} which shows we can use a much smaller covering group with much smaller orbits.  Furthermore, as Ronyai astutely observed, the action of the relevant groups on these orbits is much simpler and so enables even better algorithms than we had thought.
\end{remark}

%--------
\section{Quotients of Heisenberg groups are indistinguishable} \label{sec:invariants}

In this section, we run through a list of isomorphism invariants for finite $p$-groups of nilpotence class $2$ and determine what to
expect of these isomorphism invariants in the family of quotients of
Heisenberg groups.  The isomorphism invariants that we select are
independent in the sense that two groups with equal isomorphism
invariants of one type are not forced to have equal isomorphism
invariants of a different type.  Hence, in combination these
isomorphism invariants would seem to have a chance to distinguish
two generic $p$-groups of class $2$.

%----
\subsection{Consequences of the Camina property}
In this section, we derive some isomorphism invariants for quotients of generalized Heisenberg groups by observing these groups are special instances of Camina groups.

Recall that a group $G$ is a \emph{Camina} group if for every $g\in G-G'$, $[g,G]=G'$.    We saw in \lemref{lem:Camina} that generalized Heisenberg groups are Camina groups.  This condition transfers to all quotients by proper subgroups of $G'$.  Hence, nonabelian quotients of generalized Heisenberg groups are Camina groups.  Camina groups have received recent attention, some interesting results include \cite{Dark:on}, \cite{Mac:some}, and \cite{Mann:on}.  We use the Camina property to show that  the complex character tables of quotients of a Heisenberg group are determined solely by their order.

First, we briefly overview of representation theory and
character theory for non-experts.  Let $V$ be a finite dimensional complex vector space.   A homomorphism $\rho : G \to \GL(V)$ is an \emph{irreducible representation} if, for all $v\in V-0$, $V = \langle v(g\rho) : g \in G \rangle$. The \emph{character} $\chi_{\rho} : \{ g^G : g \in G \} \to \mathbb{C}$ afforded by $\rho$ assigns to $g \in G$ the trace of $\rho(g)$ (i.e. for each $g\in G$, $\chi_{\rho} (g^G)$ is the sum of the eigenvalues of $g\rho$, with multiplicity). The \emph{character table}, $\Irr (G)$, of $G$ is the set of characters of all irreducible representations of $G$. Finally, for groups $G$ and $H$, an \emph{isomorphism of character tables} $\Irr (G) \to \Irr (H)$ is a pair $\phi : G \to H$ and $\hat{\phi} : \Irr (H) \to \Irr (G)$ of bijections such that
\begin{align}
(\chi\hat{\phi}) (g) & = \chi (g\phi) & (\forall \chi \in \Irr (H), \forall g \in G).
\end{align}
Isomorphic groups have isomorphic character tables.  On the other hand, there are groups with isomorphic character tables that are not isomorphic.  Nevertheless,
there are incredibly deep properties of groups that can be inferred from character tables,
but that expansive subject is not our objective; for details consider \cite{Isaacs:book}.

\begin{thm}[\cite{Lewis:Camina}]\label{thm:Camina}
If $G$ and $J$ are finite Camina $p$-groups of nilpotence class $2$, then $G$ and $J$ have isomorphic character tables if and only if $[G:G'] = [J:J']$ and $|G'| = |J'|$.
\end{thm}

Moreover, the characters in question are fully described in \cite{Lewis:Camina}.
The implications of \thmref{thm:Camina} and other properties of Camina groups summarized in  \cite{Lewis:Camina} give the following list of invariants (some of which might also follow upon direct inspection of quotients of Heisenberg groups).

\begin{coro}\label{coro:Camina}
If $G$ and $J$ have the same order and are quotients of a common odd order generalized Heisenberg group $H=H_m(K)$, then the following hold:
\begin{enumerate}[(i)]
\item $G' = Z(G)$ and $J' = Z (J)$ and both are the image of $Z(H) = H'$,
\item $[G:G'] = [J:J']$ and $|G'| = |J'|$,
\item the lattice of normal subgroups of $G$ and $J$ are isomorphic (they are precisely the subgroups contained in or containing the commutator),
\item for every $g\in G - G'$ and every $h \in J - J'$, $|C_G (g)| = |C_J (h)|=[G:G']$, and
\item the character table of $G$ is isomorphic to the character table of $J$,
and if $H$ has odd order then the isomorphism of character tables also preserves power maps.
\end{enumerate}
\end{coro}

%----------
\subsection{Consequences of centroids and adjoints}
We can use the results on centroids and adjoints to determine when a quotient of a generalized Heisenberg group is directly or centrally indecomposable.

The original use of centroids of bimaps for $p$-groups was to prove the following.
\begin{thm}\cite[Theorem 1.2]{Wilson:direct-prod} \label{thm:direct-indecomp}
A $p$-group $P$ with $P' \leq Z(P)$ is  directly indecomposable if $\Cent (\Bi (P))$ is a local ring and $Z(P)$ is contained in $P'P^p$.
\end{thm}

\begin{coro}\label{coro:centroid}
The centroid of a nonabelian quotient of a generalized Heisenberg group is a field.  In particular, nonabelian quotients of generalized Heisenberg groups are directly indecomposable.
\end{coro}
\begin{proof}
Let $H/N$ be a quotient of a Heisenberg group $H$.
As $b=\Bi (H/N):V \times  V \to W$ is nondegenerate and $b (V,V)=G'=W$, it follows that $\Cent(b)$ is faithfully represented by its restriction to $\End V$.  Therefore, there is a natural embedding $\Cent(b)\hookrightarrow\Adj (b)$.  Furthermore, centroid elements commute with adjoints, and so $\Cent(b)$ embeds in the center $K$ of $\Adj (b)$.  By \corref{coro:quotient-Adj-1}, $\Adj (b)$ is central simple, and so $K$ is field.  Therefore, $\Cent(b)$ is a field, and so $\Cent(b)$ is local.  Finally, by \eqref{eq:exp}, $1=H^p\leq H'=Z(H) < H$, and so it follows that $Z(H/N)\leq (H/N)' (H/N)^p$.   By \thmref{thm:direct-indecomp} $H/N$ is directly indecomposable.
\end{proof}

The use of adjoints for $p$-groups was originally designed to understand central decompositions.  A set $\mathcal{H}$ of subgroups of a group $G$ is a \emph{central decomposition} of $G$ if $\mathcal{H}$ generates
$G$ and for all $H\in\mathcal{H}$, $[H,\langle \mathcal{H}-\{H\}\rangle]=1$ and $G\neq \langle
\mathcal{H}-\{H\}\rangle$.  Say that $G$ is \emph{centrally indecomposable} if $\{G\}$ is the only
central decomposition of $G$.  Finally, a central decomposition is \emph{fully refined} if every member
is centrally indecomposable.  For example, in a generalized Heisenberg group $H=H_m(K)$,
 for each $0\neq x \in K^m$,
\begin{align}
    H_x & = \left\{ \begin{bmatrix} 1 & tx & s\\ 0 & I_m & t' x^t\\ 0 & 0 & 1 \end{bmatrix}:
        s,t,t'\in K\right\}\cong H_1(K)
\end{align}
is a centrally indecomposable subgroup of $H_m(K)$.  If $\mathcal{X}$ is a basis for $K^m$, then
\begin{align}\label{eq:cent-decomp}
    \mathcal{H}(\mathcal{X})=\{H_x : x \in \mathcal{X}\}
\end{align}
is a fully refined central decomposition of $H$.  We now apply the following result.

\begin{thm}{\rm \cite[Theorem 4.4]{Wilson:unique-cent} with \cite[Theorem 3.8]{Wilson:algo-cent}}\label{thm:indecomp}
A $p$-group $P$ of class $2$ is centrally indecomposable if and only if $Z(P)\leq P' P^p$ and
$\Adj (\Bi (P))/J(\Adj (\Bi (P))$ is isomorphic as a $*$-ring to one of the following: for a field $K$,
\begin{description}
\item[Orthogonal] $(K,x\mapsto x)$,
\item[Unitary] $(F,x\mapsto \bar{x})$ for a quadratic field extension $F/K$ along with the field
automorphism of order $2$,
\item[Exchange] $(K \times K,(x,y)\mapsto (y,x))$, or
\item[Symplectic] $\left(M_2(K), \begin{bmatrix} a & b\\ c & d\end{bmatrix}\mapsto
\begin{bmatrix} d & -b \\ -c & a \end{bmatrix}\right)$.
\end{description}
\end{thm}

When the degree $m$ of a generalized Heisenberg group $H$ is more than $1$, we know $H$ is centrally decomposable (see \eqref{eq:cent-decomp}).  Because $H'$ is also the Frattini subgroup of $H$, if $N < H'$, then every central decomposition of $H$ induces a central decomposition of $H/N$.  So nonabelian quotients of $H_m(K)$, $|K|=p^d$ are centrally decomposable whenever $m>1$.  So suppose $m=1$, that is, that $H$ is a Heisenberg group.  By \thmref{thm:rec-q-Hei}, for every $N < H'$, $\Adj (\Bi (H/N))$ is simple of Symplectic type.  Therefore, by \thmref{thm:indecomp}, $H/N$ is centrally indecomposable if $H/N$ is indigenous to $H$.  In fact, the converse of this is true.

\begin{prop}\label{prop:cent-indecomp}
Let $H/N$ be a nonabelian quotient of a Heisenberg group $H$ over $K$.  The following are equivalent.
\begin{enumerate}[(i)]
\item $H/N$ is centrally indecomposable.
\item $\Adj (\Bi (H/N))$ is $*$-isomorphic to $M_2(K)$ with the involution of \eqref{eq:adj-j}.
\item $H/N$ is indigenous to $H$.
\end{enumerate}
\end{prop}
\begin{proof}
Suppose (i).  By \corref{coro:quotient-Adj-1}, $\Adj (\Bi (H/N))$ is $*$-isomorphic to a central simple ring with the involution of \eqref{eq:adj-j}.  Hence, by \thmref{thm:indecomp}, $\Adj (\Bi (H/N))$ is $*$-isomorphic to $M_2(L)$, for a field $L$, and $M_2(L)$ is equipped with the involution of \eqref{eq:adj-j}.  As $V=(H/N)/(H/N)'\cong H/H'$ it follows that $\dim_L V=2$ while also $\dim_K V=2$.  Hence, $K\cong L$. So (i) implies (ii).  Assuming (ii) it follows from \eqref{eq:floor} that $\lfloor H/N\rfloor$ is a Heisenberg group over $K$.    So (ii) implies (iii).  Finally, if (iii) is true, then $\Adj (\Bi (H/N))=\Adj (H)=M_2(K)$ with the involution \eqref{eq:adj-j}.  By \thmref{thm:indecomp}, $H/N$ is centrally indecomposable.
\end{proof}

Heisenberg groups can have quotients that are centrally decomposable (e.g. a Heisenberg group over a field of size $p^d$ has quotients isomorphic to $H_{d/e}(K)$, $|K|=p^e$, where $e|d$ -- these quotients are centrally decomposable unless $d=e$). It would seem that we could use the size of a fully refined central decomposition as an isomorphism invariant to distinguish some of the various quotients that could occur in \thmref{thm:main}. This requires a much deeper theorem than it may seem.  For example, there is a $2$-group of class $2$ that has fully refined central decompositions of different sizes.  However, \cite[Theorem 1.1]{Wilson:unique-cent} implies that the size of a fully refined central decomposition of a quotient of a Heisenberg group is an isomorphism invariant.\footnote{Indeed, because the adjoints of quotients $Q$ of Heisenberg groups are of Symplectic type we can further claim that the automorphism group of $Q$ acts transitively on the set of fully refined central decompositions of $Q$; cf. \cite[Corollary 6.8]{Wilson:unique-cent}.}  Nevertheless, we can dash that hope as well by arranging the orders of our groups to force them all to be centrally indecomposable, yet maintain the growth developed in \thmref{thm:main}.

\begin{coro}
Let $(n,d)$ be a pair as in \lemref{lem:good-pair}.  If $H$ is a Heisenberg group of order $p^{3d}$ and $N\leq H'$ with $[H:N]=p^n$, then $H/N$ is centrally indecomposable of symplectic type.
\end{coro}
\begin{proof}
This follows from \propref{prop:stable-q} followed by \propref{prop:cent-indecomp}.
\end{proof}

Finally, we turn to the automorphism groups of the quotients $G$ of a Heisenberg group.  For most large families of groups, it is impossible to describe the entire automorphism group of every member, and here we have not succeeded in the fullest generality.  However, we are able to describe a very large portion of the automorphism group of such a group $G$.

\begin{thm}\label{thm:aut-G}
If a group $G$ has order $p^n$ and is an indigenous quotient of a generalized Heisenberg group $H = H_m(K)$, $|K|=p^d$, then
\begin{align*}
	C_{\Aut G}(G') & \cong \Sp(2m,K) \ltimes_{\tau} \hom_{\mathbb{Z}/p}(K^{2m},\mathbb{Z}/p^{n-2md})
\end{align*}
where for each $f \in \hom_{\mathbb{Z}/p}(K^{2m},\mathbb{Z}/p^{n-2me})$ and each $\phi \in \Sp (2m,K)$, $f (\phi\tau) = \phi^{-1} f$.  Also, taking $G = H/M$, for $M < H' \cong K$, it follow that
\begin{align*}
	\Aut G/C_{\Aut G} (G') & \cong \mathbb{Z}_e\ltimes k^{\times}
\end{align*}
for some integer $e | d$ and a subfield $k$ of $K$ such that $|k|$ divides $p^n$.
\end{thm}
\begin{proof}
By \propref{prop:Baer-correspondence}, $C_{\Aut G} (G') = \Isom (\Bi (G)) \ltimes_{\tau} \hom_{\mathbb{Z}/p} (G/Z(G),G')$.  Let $V = G/Z(G)$ and $W = G'$.  Since $G$ is an indigenous quotient of $H := H_m(K)$, $V$ is isomorphic to $K^{2m}$ and $W$ is a quotient of $K$ of $\mathbb{Z}/p$-dimension $n - 2md$.  Furthermore, by Theorems~\ref{thm:rec-q-Hei}(ii) and \ref{thm:simple-adj-tensor}, $\otimes_{\Adj(\Bi (G))}$ is an alternating nondegenerate $K$-form of rank $2m$ and so $\Isom (\Bi (G))=\Isom (\otimes_{\Adj(\Bi (G))})=\Sp (2m,K)$.

Next, assume $G\cong H/M$ where $H = H_1(K)$, $|G| = p^n$, $|K| = p^d$ and $(n,d)$ satisfy \lemref{lem:good-pair}(i) and (ii).  For each $\varphi\in \Aut G$, as in \eqref{eq:Phi}, there is a $\Phi\in \Aut H$ such that $M\Phi = M$ and $\Phi|_{H/M} = \varphi$.  By \thmref{thm:aut-Hei}, $\Aut H$ acts on $H' = K$ as $\Gal(K)\ltimes (K^{\times})$.  If $\Phi|_{H'}\in K^{\times}$, then $M\Phi = Ms$ for some $s\in K^{\times}$.  Evidently $\mathbb{Z}/p\subseteq \{s\in K :  Ms\subseteq M\} = k$ is a subfield of $K$.  We show $k^{\times}$ embeds in $\Aut G$.

First, $(\Aut G)|_G'$ embeds in $\Gal(K)\ltimes k$ (observing that $\Gal(K)$ acts on $k$ because subfields of finite fields are characteristic). In particular, $G'$ is a vector space over $k$.  Also, recall from \thmref{thm:aut-Hei} that the action of $K^{\times}$ on $H$ splits with $C_{\Aut G}(G')$ and that the prescribed representation on $G/G'\cong H/H'\cong K^{2m}$ was $\rho_s : s\mapsto \begin{bmatrix} 1&0\\ 0 & s\end{bmatrix}$.  In particular, $\Sp(2m,K)$ contains $\begin{bmatrix} 0 & I_{m}\\ I_m & 0\end{bmatrix}$.\footnote{This involution interchanges two complementary maximal totally isotropic subspaces $K^m \times  0$ and $0 \times K^m$ of $V = K^{2m}$ with respect to the geometry of $j$ on $V$.} So $\Aut G|_V$ contains $\begin{bmatrix} sI_m & 0 \\ 0 &tI_m\end{bmatrix}$ for all $s,t\in k^{\times}$.  In particular, $V$ and $W$ are both $k$-vector spaces.  Indeed, we have that $|G| = [G:G']|G'|$ is a multiple of $|k|$ and that $k^{\times}$ embeds in $\Aut G$.
\end{proof}

Following \thmref{thm:aut-G}, if $G$ is a proper indigenous quotient of $H = H_m(K)$, $|K|=p^d$, and $|G|=p^n$, then $C_{\Aut G}(G')$ is determined completely by $(p,m,d,n)$.  Also, the quotient $\Aut G/C_{\Aut G}(G')\cong \mathbb{Z}_e\ltimes k^{\times}\cong \mathbb{Z}_e\ltimes \mathbb{Z}_{p^f-1}$ where $e|d$ and $f|d$.  Furthermore, $p^n$ is a multiple of $|k| = p^f$ and $f < d$ (as $G$ is not isomorphic to $H$).  So $f|n$ and $f|d$.  That severely restricts the possible outcomes.  For example, we may simply have $n$ and $d$ relatively prime, or in fact, make $d$ prime.  Therefore, it follows that $(n,d)$ satisfies \lemref{lem:good-pair} (i) and (ii), $e \in \{1,d\}$, and $f = 1$.  In particular, we have only two possible outcomes for $\Aut G/C_{\Aut G}(G')$, and this is far too small a variation to help distinguish the vast number of isomorphism types that are possible for $G$.

%-----------
\section{Closing remarks}\label{sec:closing}

\subsection{2-groups}\label{sec:2-groups}

In our first version of this article, we included quotients $G$ of Heisenberg $2$-groups.  Though some of the arguments are unchanged, there were technical flaws whose resolutions ultimately detracted from the goals set forth in our introduction.  Also, it was well-known that the isomorphism types of quotients of Heisenberg $2$-groups are determined by the character tables together with power maps (cf.  \cite{Nenciu:VZ}).  For these reasons, we opted to focus on the odd prime case.  Below we outline the different strategy needed for $2$-groups.

A group of exponent $2$ is abelian, and so, we cannot use that assumption with quotients of Heisenberg groups.  However, we can replace the need for exponent $2$ by assuming only that our group is generated by appropriate subgroups of exponent $2$.  (Many definitions below apply to odd primes as well.)

We say a group $G$ is \emph{hyperbolic} if it has abelian normal subgroups $E$ and $F$ such
that $G = EF$ and $E\cap F = Z(G)$.  (This name is motivated by the term hyperbolic as used with classical forms and has no intended relationship to hyperbolic groups in the sense of Gromov.)  The pair $(E,F)$ is a \emph{hyperbolic pair} for $G$. If $Z(G)$ splits in $E$ and $F$, then we say that $G$ is \emph{split hyperbolic}.

\begin{ex}\label{ex:Hei-hyper}
Generalized Heisenberg groups (over \emph{any} field $K$) $H = H_m(K)$ are split hyperbolic groups, e.g. they have the following split hyperbolic pair:
\begin{equation}\label{eq:Hei-pair}
E  = \left\{ \begin{bmatrix} 1 & u & s \\ 0 & I_m & 0 \\ 0 & 0 & 1 \end{bmatrix}:
 s\in K, u\in K^m\right\} \&
 F = \left\{ \begin{bmatrix} 1 & 0 & s \\ 0 & I_m & v^t \\ 0 & 0 & 1 \end{bmatrix}:
 s\in K, v\in K^m\right\}.
\end{equation}
\end{ex}

We now show that creating hyperbolic groups is easy.  The idea dates back to Brahana \cite{Brahana}.  Let $c : U \times  V \to W$ be a bimap, and define a group $\Grp_{Bra}(c)$ on $U \times V \times W$ with product
\begin{align*}
	(u,v;s)(x,y;t) & = (u+x, v+y; s+t+c(u,y) ) & (\forall (u,v;s),(x,y;t)\in U \times V \times W).
\end{align*}
Note $G := \Grp_{Bra}(c)$ is a hyperbolic group of nilpotence class $2$ with hyperbolic pair
$E = U \times 0 \times W$ and $F = 0 \times V \times W$.  If $W = c(U,V)$ and $c$ is nondegenerate, then $G' = Z(G) = 0 \times  0 \times W$ and $(E,F)$ is a split hyperbolic pair.  Observe that isotopic bimaps produce isomorphic groups.
\begin{ex}\label{ex:bi-Hei}
If $K$ is a field and $d : K^m \times K^m \to K$ is the dot-product
(i.e.: $d(u,v) = uv^t$, for all $u,v\in K^m$), then $\Grp_{Bra}(d)$ is
isomorphic to the generalized Heisenberg group of degree $m$ over
$K$.
\end{ex}

We still need to replace $\Bi$ from the Baer correspondence.  A nilpotent group $G$ of class $2$ has a hyperbolic pair $(E,F)$ if and only if $G/Z(G) = E/Z(G)\oplus F/Z(G)$ and $b (E/Z(G),E/Z(G)) = 0 = b (F/Z(G),F/Z(G))$, for $b = \Bi (G)$.  Assuming that $(E,F)$ is a hyperbolic pair for $G$, we may restrict $b$ to a second bimap:
\begin{align*}
    c &  = \Bi (G;E,F) :E/Z(G) \times F/Z(G) \to Z(G)
\end{align*}
where $c(u,v)  = b ( u,v)$ for all $u\in E/Z(G)$ and all $v\in F/Z(G)$.  As $(E,F)$ is a hyperbolic pair and $b$ is alternating, it follows for all $u,x \in E/Z(G)$ and all $v,y \in F/Z(G)$ that
\begin{equation*}
\begin{split}
	b (u+v,x+y) & = b (u,x) + b (u,y)+b (v,x) + b (v,y)
%		 = b (u,v')-b (u',v) \\
		 = c(u,y) - c(x,v).
\end{split}
\end{equation*}
Hence, $c$ determines $b$, and $c$ is nondegenerate. Unfortunately, this depends on the choice of hyperbolic pair $(E,F)$, and so, it introduces several ambiguities.  In the special case of a group $G$ where $1 = Z(G)^2 < G^2\leq G' = Z(G) < G$ (as is the case for quotients of Heisenberg $2$-groups), we have a quadratic map $q := \Qd(G) : G/Z(G) \to G'$ where $q(Z(G)u) = u^2$ for all $u\in G$.  We also observe that if $G = \Grp(c)$, then in characteristic $2$,
\begin{align*}
	(u,v;s)^2 & = (2u,2v; 2s+c(u,v)) = (0,0;c(u,v)) & (\forall (u,v;s) \in U \times V \times W).
\end{align*}
In particular, the $c$ used to define $G$ can be recovered canonically from squares.

\begin{prop}\label{prop:any-hyper}
Let $1 = Z(G)^2 < G^2 \leq G' \leq Z(G) < G$.  If $(E,F)$ is a split hyperbolic pair for a split hyperbolic group $G$, then $\Bi (G;E,F) = \Qd(G)$ as functions.  In particular, $\Bi (G;E,F)$ does not depend on the choice of $(E,F)$.
\end{prop}

Notice that the use of a quadratic map means the role of the symplectic group in our proofs is now replaced by orthogonal groups, for example \thmref{thm:aut-Hei} must be adapted.

The following correspondence of Brahana \cite{Brahana} is perhaps the earliest version of a functorial relationship between nilpotent groups and bimaps. Unlike its later generalizations by Baer, Mal'cev, Kaloujnine, and Lazard, it applies to $p$-groups without restriction on $p$ (at the cost of specializing to  hyperbolic groups).  

\begin{prop}[Brahana 1935]\label{prop:back-forth}
A group $G$ is hyperbolic if and only if $G$ is isoclinic to $\Grp_{Bra}(c)$ for a bimap
$c : U \times V \to W$.  In particular, $c$ can be chosen to be nondegenerate and
with $W = Z(G)$.  If $G$ is split hyperbolic and $G' = Z(G)$, then the isoclinism can be selected
to be an isomorphism.
\end{prop}
\begin{proof}
The reverse direction is explained above so we focus on the forward direction.

Let $(E,F)$ be a hyperbolic pair for a hyperbolic group $G_1$.
Let $c = \Bi (G_1;E,F)$ and  set $G_2 = \Grp(c) = E/Z(G_1) \times  F/Z(G_1) \times G_1'$.  As $G_2/Z(G_2)\cong E/Z(G_1)\oplus F/Z(G_1)$ and $G_2' = 0 \times 0 \times G_1'$, there are isomorphisms $\varphi : G_1/Z(G_1) \to G_2/Z(G_2)$ and $\ct{\varphi} : G_1' \to G_2'$ $G_1/Z(G) = E/Z(G)\oplus F/Z(G)$.  It follows that $(\varphi;\ct{\varphi}) : \Bi (G_1) \to \Bi (G_2)$ is a pseudo-isometry and so $G_1$ and $G_2$ are isoclinic.

If $G_1$ is split hyperbolic with split hyperbolic pair $(E,F)$, then there are subgroups $E_0\leq E$ and $F_0\leq F$ such
that $E = E_0\oplus Z(G)$ and $F = F_0\oplus Z(G)$.  Observe that $G_1 = E_0\ltimes F$.  We have canonical
isomorphisms $f : E_0 \to E/Z(G_1) \times 0 \times 0\leq G_2$ and $g : F \to 0 \times F/Z(G_1) \times Z(G_1)\leq G_2$.
Also, $(u,v)\mapsto (uf, vg)$, for $u\in E_0$ and $v\in F$, induces an isomorphism $G_1 \to G_2$.
\end{proof}

\begin{remark}
Brahana introduced his correspondence as between hyperbolic groups (our terminology) and trilinear $k$-forms, that is, functions $t : U \times V \times W \to k$ that are $k$-linear in each variable.  Notice $t$ determines a $k$-bimap $b : U \times  V \to \hom_k(W,k)$ by $b (u,v) = t(u,v,-)$.  Also, given a monomorphism $\tau : W \to \hom_k(W,k)$, a $k$-bimap $b : U \times  V \to W$ can be converted into a trilinear $k$-form $t : U \times V \times W \to k$ via $t(u,v,w) = w(b (u,v)\tau)$.  Thus our treatment above is equivalent to Brahana's.
\end{remark}

Using these tools one can derive appropriate variants of our main theorems.  However, as we mentioned at the start, these examples are not so satisfactory because there are well-known isomorphism invariants for such groups.  What we would very much like to know is a family of $2$-groups with expansive growth, a polynomial-time isomorphism test, and no obvious isomorphism invariants.  That is still an open problem.

\subsection{Our results as a `converse' to Brauer's problem}
A final consequence of our results concerns Brauer tuples.
Two groups $G$ and $H$ form a \emph{Brauer pair} if they are
nonisomorphic yet have an isomorphism between their character tables that preserves
powers.  Brauer asked if such pairs exist \cite[p. 138]{Brauer}, suggesting
that perhaps the character table considered along with powers would determine the isomorphism class of a finite group.  This was answered in the negative by Dade \cite{Dade}.  Nenciu \cite{Nenciu:VZ} showed there are no Brauer pairs of Camina $2$-groups of nilpotence class $2$ and the second author describes conditions for odd Camina $p$-groups of nilpotence class $2$ to be Brauer pairs \cite{Lewis:Brauer}.
Brauer pairs have since been generalized.  Following Eick and M\"uller in
\cite{EiMu} and Nenciu in \cite{Nenciu:tuple}, we say that the
groups $(G_1, \dots, G_t)$ form a \emph{Brauer $t$-tuple} if for all $1\leq i < j\leq t$,
$(G_i, G_j)$ is a Brauer pair.  Eick and
M\"uller proved the existence of Brauer $4$-tuples \cite{EiMu}, and
Nenciu proved the existence of $t$-tuples for
arbitrarily large $t$ in \cite{Nenciu:tuple}.

\corref{coro:Camina}(v)  and \thmref{thm:main} give Brauer $t$-tuples of exponential size $t$.  These new $t$-tuples are quite different from previous examples.
In fact, we see our result as a converse to Brauer's problem.  We give a seemingly routine set of groups that are pairwise nonisomorphic.  Should there not also be a routine explanation of why two members from the set are nonisomorphic?  

%-------------------------
\section*{Acknowledgments}
We are grateful to the referee whose comments both improved the writing and prompted us to notice our gaps for the $2$-group setting.

%---------------
%  Bibliography.

\def\cprime{$'$}

\end{document}